\documentclass[11pt,a4paper,reqno]{amsart}
\usepackage{amsmath,amssymb, amsbsy}
\usepackage{enumitem}
\usepackage{hyperref}
\usepackage{color,psfrag}
\usepackage[dvips]{graphicx}
\usepackage[arrowdel]{physics}
\usepackage{color}
\usepackage{bbm}
\usepackage{comment}
\usepackage{dsfont}
\newcommand{\noi} {\noindent}
\theoremstyle{plain}
\newtheorem{thm}{Theorem}[section]
\theoremstyle{plain}
\newtheorem{lem}[thm]{Lemma}
\newtheorem{prop}[thm]{Proposition}
\newtheorem{cor}[thm]{Corollary}

\theoremstyle{definition}

\newtheorem{rem}{Remark}[section]

\newcommand{\De} {\Delta}
\newcommand{\la} {\lambda}
\newcommand{\rn}{\mathbb{R}^{N}}
\newcommand{\bn}{\mathbb{B}^{N}}

\newcommand{\authorfootnotes}{\renewcommand\thefootnote{\@fnsymbol\c@footnote}}%

\numberwithin{equation}{section} \allowdisplaybreaks

\usepackage[text={6in,8.6in},centering]{geometry}
\begin{document}\title[Nonlinear Scalar field equation with Critical Perturbation ]
		{On a class of elliptic equations with Critical Perturbations in the hyperbolic space}

		\author[Debdip Ganguly]{Debdip Ganguly}
	\address{ Department of Mathematics, Indian Institute of Technology Delhi, Hauz Khas New Delhi 110016,  India}
	\email{debdip@maths.iitd.ac.in}
	
	\author[Diksha Gupta]{Diksha Gupta}
	\address{ Department of Mathematics, Indian Institute of Technology Delhi, Hauz Khas New Delhi 110016,  India}
	\email{dikshagupta1232@gmail.com}
	
	\author[ K.~Sreenadh]{K.~Sreenadh}
	\address{ Department of Mathematics, Indian Institute of Technology Delhi, Hauz Khas New Delhi 110016,  India}
	\email{sreenadh@maths.iitd.ac.in}

	\date{\today}
	
	\begin{abstract}
		We study the existence and non-existence of positive solutions for the following class of nonlinear elliptic problems in the hyperbolic space
		\begin{equation*}
		-\Delta_{\bn} u-\la  u=a(x)u^{p-1} \, + \, \varepsilon u^{2^*-1} \,\;\;\text{in}\;\mathbb{B}^{N}, \quad
			u \in H^{1}{(\mathbb{B}^{N})},
	\end{equation*}
	where $\mathbb{B}^N$ denotes the hyperbolic space,  $2<p<2^*:=\frac{2N}{N-2}$, if $N \geqslant 3; 2<p<+\infty$, if $N = 2,\;\lambda < \frac{(N-1)^2}{4}$, and $0< a\in L^\infty(\mathbb{B}^N).$ We first prove the existence of a positive radially symmetric ground-state solution for $a(x) \equiv 1.$ Next, we prove that for $a(x) \geq 1$, there exists a ground-state solution for $\varepsilon$ small. For proof, we employ ``conformal change of metric" which allows us to transform the original equation into a singular equation in a ball in $\mathbb R^N$. Then by carefully analysing the energy level using blow-up arguments, we prove the existence of a ground-state solution. Finally, the case $a(x) \leq 1$ is considered where we first show that there is no ground-state solution, and prove the existence of a \it bound-state solution \rm (high energy solution) for $\varepsilon$ small. We employ variational arguments in the spirit of Bahri-Li to prove the existence of high energy-bound-state solutions in the hyperbolic space.

\end{abstract}

\date{\today}
\maketitle


\section{Introduction}	
	
\noi In this paper, we investigate the existence of solutions for the following  class of critical elliptic problem in the hyperbolic space $\mathbb{B}^{N}$
\begin{equation*}
-\Delta_{\bn} u-\la  u=a(x)u^{p-1}+\varepsilon u^{2^*-1} \text { in } \bn,\; u>0  \text { in } \bn, \; u \in H^1(\bn),  \label{mainprob} \tag{$P_{\varepsilon}$}
\end{equation*}
where   $2<p<2^*:=\frac{2N}{N-2}$, if $N \geqslant 3; 2<p<+\infty$, if $N = 2,\;\lambda < \frac{(N-1)^2}{4}$, $\varepsilon$ is a real parameter, $H^{1}\left(\mathbb{B}^{N}\right)$ denotes the Sobolev space 
on the disc model of the hyperbolic space $\mathbb{B}^{N},$ $\Delta_{\mathbb{B}^{N}}$ denotes the Laplace Beltrami operator on $\mathbb{B}^{N}.$ Further, let $\mu$ denotes the hyperbolic volume measure, and $d(x,0)=\log(\frac{1+|x|}{1-|x|})$ is the hyperbolic distance of $x$  from $0$. Moreover, we investigate the existence of solutions for appropriately chosen $\varepsilon$ under the following hypotheses separately:
\begin{equation*}
	\begin{aligned}
	\left(\mathbf{A}_{1}\right):& \;\;a(x) \geq 1\;\;\forall x \in \bn,\;\;\mu( \{ x : a(x) \not\equiv 1\}) > 0, \;\; a \in L^{\infty}(\bn) 
		\text{ and } a(x) \rightarrow 1 \\
  &  \text{ as }d(x,0) \rightarrow \infty.\\
		\left(\mathbf{A}_{2}\right):& \;\; a(x) \in(0,1]\;\; \forall x \in \mathbb{B}^{N}, \;\; \mu( \{ x : a(x) \not\equiv 1\}) > 0, \;\; \inf_{x \in \bn} a(x) > 0, 
		\text { and } \\
		& \;\;a(x) \rightarrow 1 \text { as } d(x,0) \rightarrow \infty.\\
		\left(\mathbf{A}_{3}\right): &
		\;\; a(x) \equiv 1 \;\; \forall x \in \mathbb{B}^{N}.
	\end{aligned}
\end{equation*}
Further, let us prescribe the following assumption on the parameter $\lambda :$
\begin{equation}\label{lambda}\lambda \in \begin{cases}  \left( -\infty, \frac{2(p+1)}{(p+3)^2}\right], & N=2, \\ \left(-\infty, \frac{(N-1)^2}{4} \right), & N\ge 3.\end{cases}\end{equation}
Here, $\frac{(N-1)^2}{4}$ is the bottom of the $L^2-$ spectrum of $-\Delta_{\bn}.$

\medskip 

We recall that the solutions of \eqref{mainprob} are the critical points of the corresponding energy functional $ E_{\varepsilon}: H^1(\bn) \rightarrow \mathbb R$ defined as
	\begin{equation} \label{enma}
E_{\varepsilon}(u)=\frac{1}{2} \int_{\bn}\left(|\nabla_{\bn} u|^2- \la u^2\right) \mathrm{~d} V_{\bn}-\frac{1}{p} \int_{\bn}a(x)|u|^p \mathrm{~d} V_{\bn}-\frac{\varepsilon}{2^*} \int_{\bn}|u|^{2^*} \mathrm{~d} V_{\bn}.
\end{equation}	
Then $E_{\varepsilon}$ is a well-defined $C^1$ functional on $H^1(\bn).$ We use variational, refined energy estimates and blow-up arguments to prove the existence of solutions. The intriguing nature of the problem  is related to the fact that the equation 
	\eqref{mainprob} is non-compact, so standard variational methods fail. The problem studied in this article is in continuation to our study on scalar-field type equations on the hyperbolic space (see \cite{DDS, DDS1}) where we only dealt with the purely subcritical problem, i.e., when $\varepsilon =0,$ the unperturbed problem. In the subcritical case, the variational problem lacks compactness because of the {\it hyperbolic translation\rm} (see section~\ref{sec2} for more details), and so it cannot 
be solved by the standard minimization method. Moreover, in \cite{BS}, a detailed analysis of the Palais-Smale decomposition is performed. One can easily see that if $U$ is 
a solution of \eqref{mainprob}, with $\varepsilon=0$ and $a(x) \equiv 1,$ then
\[u := U\circ\tau, \quad \mbox{for} \ \tau \in I(\bn),
\]
\noi where $I(\bn)$ is the group of isometries on the hyperbolic space,   is also a solution. Hence if we define a sequence by varying $\tau_n\in I(\bn),$ then for a Palais-Smale (PS) sequence $u_{n}$, $u_{n}\circ\tau_n$ is also a  Palais-Smale (PS) sequence for \eqref{mainprob} with $\varepsilon=0$ and $a(x) \equiv 1$. In fact, it was shown in 
\cite[Theorem~3.3]{BS} that in the subcritical case, i.e., when $2< p < \frac{2N}{N-2},$ non-compact PS sequences
 are made of finitely many sequences of the form $U\circ \tau_n$.

\medskip 

The problem we considered in this article also has a critical nonlinearity for $\varepsilon >0$. The Palais-Smale decomposition established in \cite[Theorem~3.3]{BS} reveals that for critical exponent problems in the hyperbolic space, loss of compactness can happen along two different profiles, one along the hyperbolic translations and the other along concentration of Aubin-Talenti bubble (locally). This makes the problem \eqref{mainprob} very fascinating. There will be an interplay between subcritical and critical nonlinearity. Indeed, some concentration phenomena can happen owing to critical nonlinearity. Although we will only consider $\varepsilon$ small enough so that the critical nonlinearity can be seen as a perturbation of the subcritical problem. Before analysing the difficulties and methodology we adopt to restore compactness, we first discuss the ``state of the art" of such problems when posed in the Euclidean space. 

\medskip 

There has been intensive research over the past few decades on  \eqref{mainprob} when $\varepsilon =0$ in the Euclidean space after the seminal papers by Berestycki-Lions \cite{BL1, BL2}, Bahri-Berestycki \cite{BB1}, Bahri-Li \cite{Bahri-Li}, Bahri-Lions \cite{Bahri}. Many authors have contributed to a much deeper understanding of the problem in the framework of existence and multiplicity, we name a few,  e.g.,\cite{AD,  Adachi, AD2, CZ, CW, DN, LJ, AM, MMP, XZ}, and this list is far from being complete. The fundamental challenge in dealing with such problems in unbounded domains in $\mathbb{R}^N$ is the lack of compactness, even in the subcritical case, thus preventing the typical variational approaches from succeeding. As a result, various authors have studied these equations by deploying various conditions on $a(x)$ and presented new tools and methodologies to overcome this difficulty. For example, if $a(x)= a(|x|)$, the compactness of the embedding $H_r\left(\mathbb{R}^N\right)$, the subspace of $H^1\left(\mathbb{R}^N\right)$ consisting of radially symmetric functions into $L^p\left(\mathbb{R}^N\right),\;p \in\left(2,2^*\right)$ helps to restore the usage of standard variational arguments(\cite{BL1, BL2}). However, if the symmetry restriction on $a(x)$ is dropped, the problem becomes more exciting and complex. In particular, the authors in \cite{Bahri-Li, Bahri} have established the existence of positive solutions by considering the asymptotic condition on $a(x)$, i.e., $a(x) \rightarrow a_{\infty}$ as $|x| \rightarrow \infty$ and appropriate decay estimates. They carefully examined the levels of failure of PS condition, then searched for high energy solutions when ground state solution did not exist and used delicate variational and topological arguments. The subject of the multiplicity of solutions has also been investigated in \cite{CG4, CG3, RD2, RD1, VR} under different circumstances like some suitable assumption on $|a(x)- a_{\infty}|$ or some order relation between $a(x)$ and $a_{\infty},$  i.e., $a(x)$ goes to $a_{\infty}$ from above or below or a certain periodicity assumption on $a$. Furthermore, in \cite{C5} a more general equation $-\Delta u+\alpha(x) u=\beta(x)|u|^{p-1} u$ has been studied in $\mathbb{R}^N$ where $\alpha$ and $\beta$ are positive functions such that $\lim _{|x| \rightarrow \infty} \alpha(x)=a_{\infty}>0$ and $\lim _{|x| \rightarrow \infty} \beta(x)=b_{\infty}>0$. They have proven the existence and non-existence of ground state solutions under a variety of hypotheses like $\alpha(x) \rightarrow a_{\infty}$ from below and $\beta(x) \rightarrow b_{\infty}$ from above and vice-versa, $\alpha(x)$ decays faster or slower than $\beta(x).$ 

\medskip

Further, the scalar field equations in $\mathbb{R}^{N}$ involving the critical exponent provide an even greater mathematical challenge because of the loss of compactness in two profiles, translation in $\mathbb{R}^N$ and the presence of the critical exponent. The following class of Dirichlet problems in different domains in $\mathbb{R}^N$ with sufficiently smooth boundary conditions has been the focus of much research over the past several years 
\begin{equation*}
-\Delta u- \eta u=|u|^{2^*-2}u  \quad \mbox{in} \quad \Omega
\end{equation*}
where $\eta$ is a real parameter, and $\Omega$ is a domain in $\mathbb{R}^N.$ Brezis-Nirenberg, in their commendable work {\cite{BN}} have shown the existence and non-existence of positive solutions in bounded domains for $\eta>0$. They identified the first critical level below which compactness can be restored with the help of Aubin-Talenti functions, popularly known as \emph{bubbles}. On the other hand, for $\eta \leq 0,$ the shape of the domain comes into play, and it is well known using the Pohozaev identity that solutions cease to exist for star-shaped domains. Following that, attempts were made to discover solutions either by altering the domain's shape (\cite{L1, L2, L3, L4}) or experimenting with the lower order terms (\cite{S2, S3}). Also, see \cite{BN1, C1, C2} and references therein.

\medskip

Thereafter the authors in \cite{SR} studied a problem involving critical and subcritical non-linearities (also, see \cite{E1}). This motivated us to study the related problem \eqref{mainprob} in the hyperbolic space. It makes sense that subcritical analysis (\cite{DDS, DDS1, MS}) for $\varepsilon=0$ cannot be applied given the loss of compactness in two profiles, one of which is caused by the presence of the critical exponent. Moreover, because of the hyperbolic translations, the other profile can be owed to the following limiting problems
\begin{equation*}
-\Delta_{\bn} u - \la u=|u|^{p-2} u \text { in } \bn, \;\; \\
u \in H^1\left(\bn\right)
\label{infprob} \tag{$P_{\infty}$}
\end{equation*}
and 
\begin{equation*}
-\Delta_{\bn} u- \la u=|u|^{p-2} u+\varepsilon|u|^{2^*-2} u \text { in } \bn,\;\;
u \in H^1\left(\bn\right)
\label{limprob} \tag{$P_{\varepsilon,\infty}$}.
\end{equation*}
Since the  PS decomposition for the problems of the type \eqref{limprob} is yet to be discovered, we can no longer employ the standard tools and techniques.
\newpage
\noindent
\subsection{Methodologies and strategy}

	The Nehari set for a functional $J$ defined on a function space $X$ is defined as
	\[\mathbbmss{N} = \{u\in X, \; J^\prime(u)[u]=0\}.\]
	It is easy to show that this set is a manifold for a large class of functionals associated with elliptic problems such as \eqref{mainprob}, \eqref{limprob}, \eqref{infprob}. Additionally, it can also be proven that the functionals are bounded below on this Nehari manifold.  Suppose $E_\infty$ and $\mathcal{N}_\infty$ are the functional and Nehari manifold, respectively, associated with the problem \eqref{infprob}. Then the minimization problem (see \cite{MS})
	$$m=\inf _{\mathcal{N}_{\infty}} E_\infty$$
has a solution, and $m$ is achieved by some $w \in \mathcal{N}_{\infty}$, thus solving \eqref{infprob}. To be precise, authors in \cite{MS} established that in the subcritical case, and for $p > 2,$ if $N =2$ and $2 < p < 2^{\star}$ if $N \geq 3,$  the problem \eqref{infprob} has a positive solution if and only if $\lambda < \frac{(N-1)^2}{4}.$ These 
positive solutions are also shown to be unique up to hyperbolic isometries, except possibly for $N =2$ and $\lambda > \frac{2(p+1)}{(p+3)^2}.$\\
The above discussion and the hypotheses $\left(\mathbf{A}_{1}\right)$ and $\left(\mathbf{A}_{2}\right)$ indicate that the corresponding limiting problems will play a vital role in studying \eqref{mainprob}. We shall recall from \cite{BS} that the solutions to the following problem can be attributed to the loss of compactness due to the critical exponent problem
\begin{equation*}
    -\Delta V=|V|^{2^*-2} V, \quad V \in D^{1,2}\left(\mathbb{R}^N\right).
    \label{concprob} \tag{$CP_{\infty}$}
\end{equation*}
\noi
We know that \eqref{concprob} and \eqref{infprob} have been thoroughly and extensively studied in \cite{BN, MS}, respectively. However, to our knowledge, $(P_{\varepsilon, \infty})$ still needs to be explored. So we first establish the existence of its solutions, particularly the ground state solution (refer Section \ref{sec3}) using the standard variational methods. For this, we define the functional $E_{\varepsilon, \infty}:H^1\left(\bn\right) \rightarrow \mathbb{R}^N$ corresponding to  $(P_{\varepsilon, \infty})$ as
\begin{equation*}
	E_{\varepsilon, \infty}(u)=\frac{1}{2} \int_{\bn}\left(|\nabla_{\bn} u|^2- \la u^2\right) \mathrm{~d} V_{\bn}-\frac{1}{p} \int_{\bn}|u|^p \mathrm{~d} V_{\bn}-\frac{\varepsilon}{2^*} \int_{\bn}|u|^{2^*} \mathrm{~d} V_{\bn},
\end{equation*}
and $\mathcal{N}_{\varepsilon, \infty}$ denotes the associated  Nehari manifold.
%
We address the following minimization problem 
\begin{equation*}
m_{\varepsilon}:=\inf _{\mathcal{N}_{\varepsilon, \infty}} E_{\varepsilon, \infty},
\end{equation*}
 and exploit the radial symmetry of \eqref{limprob} , then make use of the compactness of embedding and finally establish the existence of solution using the Ekeland Variational principle. The solution thus obtained by solving such a minimization problem is referred to as a \emph{ground-state solution}.
Then we move on to search for the solutions to our main problem \eqref{mainprob}. Also, as the uniqueness and the decay estimates on the solutions of the problem at infinity ($P_{\varepsilon, \infty}$) are still unknown, we fail to use the techniques used in \cite{Bahri-Li, DDS, DDS1} to establish the solutions of \eqref{mainprob}. However, we are able to recover the compactness under the hypothesis $\left(\mathbf{A}_{1}\right)$ below a level. In restoring the compactness, we test the sequence of scaled solutions of \eqref{concprob} on $E_{\varepsilon, \infty}$, but this is not feasible since there is no scaling in the hyperbolic space. Thus to perform this blow-up analysis, we conformally transformed the problem \eqref{limprob} to $\rn$ (refer Section \ref{sec2}). Then we obtained a series of estimates for the required integrals in $\rn$, which resulted in the restoration below the level $m_{\varepsilon}$ (Proposition \ref{propa1}). Additionally, we want to point out to the readers that these estimations are valid for $N \geq 4$ and $\frac{N(N-2)}{4}< \la < \frac{(N-1)^2}{4}$. Further, note that we obtained the following two estimates on $m_{\varepsilon}$
$$m_{\varepsilon}<\frac{1}{N} S^{N / 2}\left(\frac{1}{\varepsilon}\right)^{\frac{N-2}{2}} \text {for $\varepsilon$ small} $$ where $S$ is the best Sobolev constant that occurs in the Sobolev inequality in $\rn$, and 
$$\lim _{\varepsilon \rightarrow 0} m_{\varepsilon}=m.$$
Also, in \eqref{a1.1}, we deduce that $m_{\epsilon} \leq m$.
Therefore, in a way, we can say that we retrieved the compactness (for small $\varepsilon$) below the level where we have managed to avoid all the anticipated {\it{bubbles}}. Finally, with the help of the solution of \eqref{limprob} as determined in Theorem \ref{limthm} and the restored compactness, we find the ground state solution of \eqref{mainprob} (Theorem \ref{grst}) under the hypothesis $\left(\mathbf{A}_{1}\right)$. After that, we look for positive solutions of \eqref{mainprob} assuming $\left(\mathbf{A}_{2}\right)$. However, we prove the non-existence of ground state solution in this case (Proposition \ref{prop2ab}). Thus we look for high-energy solutions by assuming a decay estimate on $a(x)$ (Theorem \ref{aless}). We call such a solution as a \emph{high energy-bound-state solution} because it has an energy level above the level of the ground state and is located in an interval. The extra assumption on $a(x)$ helps regain the compactness locally (Proposition \ref{propps}) using the problem studied in our previous work \cite{DDS}, i.e., $\eqref{mainprob}$ when $\varepsilon =0.$ For this, the crucial step is to define an appropriate barycentric map (\cite{B2, B1}) to prove some auxiliary lemmas for the existence of bound state high-energy solutions. The usual barycentric map in $\rn$ enjoys a nice property under the action of translation. However, the highly non-linear nature of hyperbolic translation makes it difficult to achieve such a characteristic in our context. We conclude this article by proving the existence of bound state high-energy solutions by delicately applying energy estimates, barycentric maps and topological degree arguments.

\subsection{Main results}
Now we shall discuss and state our main results in this article. We shall prove the existence of solutions under the assumptions $\left(\mathbf{A}_{1}\right), \left(\mathbf{A}_{2}\right)$ and $\left(\mathbf{A}_{3}\right).$ First, we start with the simplest case, i.e., when $a(x)$ satisfies $\left(\mathbf{A}_{3}\right).$

\begin{thm} \label{limthm}
Let $a(x)$ satisfies $\left(\mathbf{A}_{3}\right),$ i.e., $a(x) \equiv 1$ for all $x \in \bn.$ Then there exists $\varepsilon_0>0$ such that for any $\varepsilon \in\left(0, \varepsilon_0\right)$ the problem \eqref{limprob} has a positive radially symmetric ground-state solution $w_{\varepsilon}$.
\end{thm}


Then, assuming $a(x) \geq 1,$ we prove the following result concerning the existence of a ground-state solution:

\begin{thm} \label{grst}
Let $a(x) \in C\left(\mathbb{B}^{N}\right) \cap L^{\infty}(\bn)$ satisfies $\left(\mathbf{A}_{1}\right).$ Assume $N \geq 4$ and 
$\frac{N(N-2)}{4} < \lambda < \frac{(N-1)^2}{4}.$
Then there exists $\varepsilon^\prime>0$ such that the problem \eqref{mainprob} has a ground-state solution for every $\varepsilon \in\left(0, \varepsilon^\prime\right)$. \label{thm1}
\end{thm}

\medskip


Now an obvious question arises whether a ground-state solution does exist under the assumption $\left(\mathbf{A}_{2}\right)$ for the problem \eqref{mainprob}. In fact, in Proposition \ref{prop2ab}, we prove the non-existence of ground-state solutions for \eqref{mainprob}.
Furthermore, when $a(x) \leq 1,$  a bound-state solution exists in this case, particularly a high-energy solution in the spirit of Bahri-Li. We shall borrow the ideas of Bahri-Li in their seminal paper \cite{Bahri-Li} to establish the existence of a bound-state solution. However, we shall describe later the many nontrivial difficulties that arise in the hyperbolic space to achieve solutions compared to the Euclidean case. In particular, we prove the following theorem:


\begin{thm} \label{aless}
 Let $a(x) \in C\left(\mathbb{B}^{N}\right)$ satisfies $\left(\mathbf{A}_{2}\right).$  In addition, assume that $a(x)$ also satisfies 
	\begin{equation}
		a(x) \geqslant 1 -\operatorname{C \, exp}(-\delta \, d(x,0)) \quad \forall x \in \mathbb{B}^{N}, \label{acond}
	\end{equation}
	for some positive constants $C$ and $\delta.$ Then there exists $\widehat{\varepsilon}>0$ such that for any $0<\varepsilon<\widehat{\varepsilon}$ the problem \eqref{mainprob} has at least one positive solution, that is a high energy-bound-state solution.
\end{thm}

\medskip 


\noi The paper is organized as follows: In Section~\ref{sec2}, we introduce some of the notations, geometric definitions, and preliminaries concerning the hyperbolic space and derive a conformal equivalent problem on the Euclidean ball. 
Section~\ref{sec3} begins with the proof of Theorem~\ref{limthm} and further establishes auxiliary propositions, which include an upper estimate on $m_{\varepsilon},$ Palais-Smale decomposition and finally completes the proof of Theorem~\ref{grst}. Finally, Section~\ref{sec4} is devoted to the proof of Theorem~\ref{aless}.


\medskip

\section{Notations and Functional Analytic Preliminaries} \label{sec2}
In this section, we will introduce some of the notations and definitions used in this
paper and also recall some of the embeddings
related to the Sobolev space on the hyperbolic space.  

\medskip

\noi We will denote by $\bn$ the disc model of the hyperbolic space, i.e., the unit disc
equipped with 
the Riemannian metric $g_{\bn} := \sum\limits_{i=1}^N \left(\frac{2}{1-|x|^2}\right)^2 \, {\rm d}x_i^2$. The Euclidean unit ball $B(0,1):= \{x \in \mathbb{R}^N: |x|<1\}$ equipped with the Riemannian metric
\begin{align*}
	{\rm d}s^2 = \left(\frac{2}{1-|x|^2}\right)^2 \, {\rm d}x^2
\end{align*}
constitute the ball model for the hyperbolic $N$-space, where ${\rm d}x$ is the standard Euclidean metric and $|x| = \left(\sum_{i=1}^Nx_i^2\right)^{1/2}$ is the standard Euclidean length. To simplify our notations, we will denote $g_{\bn}$
by $g$.
The corresponding volume element is given by $\mathrm{~d} V_{\mathbb{B}^{N}} = \big(\frac{2}{1-|x|^2}\big)^N {\rm d}x, $ where ${\rm d}x$ denotes the Lebesgue 
measure on $\rn$.

\medskip 

\noi {\bf Hyperbolic distance on $\bn$.} The hyperbolic distance between two points $x$ and $y$ in $\bn$ will be denoted by $d(x, y).$ For the hyperbolic distance between
$x$ and the origin we write 

$$
\rho := \, d(x, 0) = \int_{0}^{r} \frac{2}{1 - s^2} \, {\rm d}s \, = \, \log \frac{1 + r}{1 - r},
$$
where $r = |x|$, which in turn implies that  $r = \tanh \frac{\rho}{2}.$ Moreover, the hyperbolic distance between $x, y \in \bn$ is given by 

$$
d(x, y) = \cosh^{-1} \left( 1 + \dfrac{2|x - y|^2}{(1 - |x|^2)(1 - |y|^2)} \right).
$$
It easily follows that a subset $S$ of $\bn$ is a hyperbolic sphere in $\bn$ if and only if $S$ is a Euclidean sphere in $\mathbb{R}^N$ and contained in $\bn$, probably 
with a different centre and different radius, which can be computed. Geodesic balls in $\bn$ of radius $r$ centred at the origin will be denoted by

$$
B_{r}(y) : = \{ x \in \bn : d(x, y) < r \}.
$$

\noi We also need some information on the isometries of $\bn$. Below we recall the
definition of a particular type of isometry, namely the hyperbolic translation. For
more details on the isometry group of $\bn$, we refer to \cite{RAT}.

\medskip 

\noi {\bf Hyperbolic translation.} For $b \in \bn,$ define 

\begin{equation}\label{hypt}
	\tau_b(x) = \dfrac{(1 - |b|^2)x + (|x|^2 + 2x.b + 1)b}{|b|^2 |x|^2 + 2x.b + 1},
\end{equation}
then $\tau_b$ is an isometry of $\bn$ with $\tau_b(0) = b.$ The map 
$\tau_b$ is called the hyperbolic translation of $\bn$ by $b.$ It can also be seen that 
$\tau_{-b} = \tau_b^{-1}.$

\medskip

\noi The hyperbolic gradient $\nabla_{\bn}$ and the hyperbolic Laplacian $\De_{\bn}$ are
given by
\begin{align*}
	\nabla_{\bn}=\left(\frac{1-|x|^2}{2}\right)^2\nabla,\ \ \ 
	\De_{\bn}=\left(\frac{1-|x|^2}{2}\right)^2\De + (N-2)\frac{1 - |x|^2}{2} \, \langle x, \nabla \,\rangle.
\end{align*}

\medskip

\noi{\bf A sharp Poincar\'{e}-Sobolev inequality.} (see \cite{MS})

\medskip

\noi{\bf Sobolev Space :} We will denote by ${H^{1}}(\bn)$ the Sobolev space on the disc
model of the hyperbolic space $\bn$, equipped with norm $\|u\|=\left(\int_{\mathbb{B}^N} |\nabla_{\mathbb{B}^{N}} u|^{2}\right)^{\frac{1}{2}},$
where  $|\nabla_{\bn} u| $ is given by
$|\nabla_{\bn} u| := \langle \nabla_{\bn} u, \nabla_{\bn} u \rangle^{\frac{1}{2}}_{\bn} .$ \\
\noi
For $N \geq 3$ and every $p \in \left(1, \frac{N+2}{N-2} \right]$ there exists an optimal constant 
$S_{N,p} > 0$ such that
\begin{equation*}
	S_{N,p} \left( \int_{\mathbb{B}^{N}} |u|^{p + 1} \mathrm{~d} V_{\mathbb{B}^{N}} \right)^{\frac{2}{p + 1}} 
	\leq \int_{\mathbb{B}^N} \left[|\nabla_{\mathbb{B}^{N}} u|^{2}
	- \frac{(N-1)^2}{4} u^{2}\right] \, \mathrm{~d} V_{\mathbb{B}^{N}},
\end{equation*}
for every $u \in C^{\infty}_{0}(\mathbb{B}^{N}).$ If $ N = 2$, then any $p > 1$ is allowed.

\noi A basic information is that the bottom of the spectrum of $- \Delta_{\bn}$ on $\bn$ is 
\begin{equation}\label{firsteigen}
	\frac{(N-1)^2}{4} = \inf_{u \in H^{1}(\bn)\setminus \{ 0 \}} 
	\dfrac{\int_{\bn}|\nabla_{\bn} u|^2 \, \mathrm{~d} V_{\mathbb{B}^{N}} }{\int_{\bn} |u|^2 \, \mathrm{~d} V_{\mathbb{B}^{N}}}. 
\end{equation}

\begin{rem}
	A  consequence of \eqref{firsteigen} is that if $\lambda < \frac{(N-1)^2}{4},$ then

	$$
	\left\|u\right\|_{H_{\lambda}} := \left\|u\right\|_{\lambda} := \left[ \int_{\bn} \left( |\nabla_{\bn} u|^2 - \lambda \, u^2 \right) \, \mathrm{~d} V_{\mathbb{B}^{N}} \right]^{\frac{1}{2}},
	$$
	is a norm, equivalent to the $H^1(\bn)$ norm and the corresponding inner product is given by $\langle u, v\rangle_{H_{\lambda}}.$\\
Also, throughtout the article, we use $|\cdot|_p$ to denote the $L^p(\bn)$ norm.
\begin{rem}
	It is interesting to note that there exists $c>0$ independent of small $\varepsilon$ such that
	\begin{equation}
		\|u\|_{\la} \geq c \quad \forall u \in \mathcal{N}_{\varepsilon, \infty}. \label{ub}
	\end{equation}
	Indeed using the Poincar\'{e}-Sobolev inequality on the hyperbolic space, we have
	\begin{equation*}
		0=\|u\|^2_{\la}-|u|_p^p-\varepsilon|u|_{2^*}^{2^*} \geq\|u\|^2_{\la}-c_1\|u\|^p_{\la}-c_1 \varepsilon\|u\|^{2^*}_{\la}, \quad \forall u \in \mathcal{N}_{\varepsilon, \infty}.
	\end{equation*}
\end{rem}
\noindent

 \medskip

\noi{\bf Conformal change of metric.} We want to conformally transform the problem \eqref{limprob} to the Euclidean Space. For that, define $P_{1, \bn}:=-\Delta_{\bn}+$ $\frac{(N-2)}{4(N-1)} S_{\bn}=-\Delta_{\bn}-\frac{N(N-2)}{4}$ is the first order conformally invariant Laplacian operator where $S_{\bn}:=-N(N-1)$ is the scalar curvature of ${\bn}$. Therefore, for a conformal change in metric $\tilde{g}=e^{2 \psi} g_{\bn}$, we have $P_{1, \tilde{g}}(u)=e^{-\left(\frac{N}{2}+1\right) \psi} P_{1, \bn}\left(e^{\left(\frac{N}{2}-1\right) \psi} u\right)$ for every smooth function $u$. As a consequence of the Poincar\'{e} metric being conformal to the Euclidean metric with $\psi(x)=\ln \left(\frac{1-|x|^2}{2}\right)$ we can transform \eqref{limprob} to the Euclidean space as follows: Let $H_0^1\left(\mathrm{B}^N\right)$ is the Sobolev space on $\mathrm{B}^N$ characterized by zero traces on the boundary $\partial \mathrm{B}^N$ where $\mathrm{B}^N$ is the open Euclidean ball with centre at the origin and unit radius. Suppose $u$ be a solution to \eqref{limprob}. Set $\varphi:= $ $\left(\frac{2}{1-|x|^2}\right)^{\frac{N-2}{2}}$, then $v:=\varphi u$ solves
\begin{equation}
-\Delta v-b_{\la}(x) v= c_{p}(x) v^{p-1}+ \varepsilon v^{2^{\star}-1} , \quad v \in H_0^1\left(\mathrm{B}^N\right), \label{conprob}
\end{equation}
where $b_{\la}(x)=\frac{4 \lambda-N(N-2)}{\left(1-|x|^2\right)^2}$, $c_{p}(x)= \left(\frac{2}{1-|x|^2}\right)^{N-{\left(\frac{N-2}{2}\right)p}}$. Observe that $b_{\la}(x)>0$ in $\mathrm{B}^N$ whenever $\lambda>\frac{N(N-2)}{4}$, and $N-\left(\frac{N-2}{2}\right)p > 0$ for $p< 2^{*}$.
\end{rem}

\medskip


\section{Existence of a ground-state solution}\label{sec3}

This section is devoted to the existence of solutions of \eqref{mainprob} when the potential $a(x)$ satisfies $\left(\mathbf{A}_{1}\right)$
and $\left(\mathbf{A}_{3}\right).$ We first begin with the simplest case, when $a(x) \equiv 1.$ The proof is a straightforward adaption of standard variational arguments (\cite{SR}) in the hyperbolic setting and restoring compactness for (hyperbolic) radial functions in the subcritical
case. Let us recall a Strauss-type lemma in the hyperbolic space (see \cite[Theorem~3.1]{BS}): 

\medskip 
Let $H^1_{r}(\bn)$ denotes the subspace, 

$$
H^1_{r}(\bn) :=  \{ u \in H^1(\bn) : \mbox{ $u$ is radial}\}.
$$

\medskip 

\begin{lem}[\cite{BS}] \label{rc}
The embedding $H^1_{r}(\bn) \hookrightarrow L^p(\bn)$ for $2 < p < 2^{\star}$ is compact.
\end{lem}

\medskip

\begin{proof}[\bf Proof of Theorem~\ref{limthm}]
First, we obtain that $m_{\varepsilon} \leq m \;\forall \varepsilon>0$ as follows: let $\tau_{\varepsilon}>0$ be such that $\tau_{\varepsilon} w \in \mathcal{N}_{\varepsilon, \infty}$, then
\begin{equation} 
m_{\varepsilon} \leq E_{\varepsilon, \infty}\left(\tau_{\varepsilon} w\right) \leq E_{\infty}\left(\tau_{\varepsilon} w\right) \leq E_{\infty}(w)=m, \label{a1.1}
\end{equation}
where the first and second inequality follows from the definition of $m_{\varepsilon}, E_{\varepsilon, \infty}, E_{\infty}$, and the last inequality follows from $w$ being a unique (upto hyperbolic translations), positive radial solution of \eqref{infprob}.\\
We confine our analysis to the following sets in order to solve the minimization problem for $m_\varepsilon$
$$H_r^1\left(\bn\right), \;\;\mathcal{N}_{r}^{\varepsilon}=\mathcal{N}_{\varepsilon, \infty} \cap H_r^1\left(\bn\right).$$
Let $\left\{u_n^\varepsilon\right\}_n$ in $\mathcal{N}_{r}^{\varepsilon}$ be a minimizing sequence, i.e., 
\begin{equation}
\left\|u_n^\varepsilon\right\|_{\la}^2 =\left|u_n^\varepsilon\right|_p^p+\varepsilon\left|u_n^\varepsilon\right|_{2^*}^{2^*}, \label{a.1}
\end{equation}
\begin{equation}
E_{\varepsilon, \infty}\left(u_n^\varepsilon\right)=\left(\frac{1}{2}-\frac{1}{p}\right)\left\|u_n^\varepsilon\right\|_{\la}^2+\left(\frac{1}{p}-\frac{1}{2^*}\right) \varepsilon\left|u_n^\varepsilon\right|_{2^*}^{2^*}=m_{\varepsilon}+o(1).\label{a.2}
\end{equation}
Using the inequalities \eqref{a1.1} and \eqref{a.2}, we get
\begin{equation}
\left\|u_n^\varepsilon\right\|_{\la}^2 \leq\left(\frac{1}{2}-\frac{1}{p}\right)^{-1} m_{\varepsilon}+o(1) \leq\left(\frac{1}{2}-\frac{1}{p}\right)^{-1} m+o(1).\label{a.3}
\end{equation}
Notice that from $\eqref{a.1},\eqref{ub},\eqref{a.3}$ and the Sobolev embedding theorem, we get the existence of $\varepsilon_0>0$ such that, for all $n \in \mathbb{N}$,
\begin{equation}
\left|u_n^\varepsilon\right|_p^p \geq\left\|u_n^\varepsilon\right\|_{\la}^2-c\; \varepsilon\left\|u_n^\varepsilon\right\|_{\la}^{2^*} \geq \text { constant }>0 \quad \forall \varepsilon \in\left(0, \varepsilon_0\right).\label{neq}
\end{equation}

Further, Lemma \eqref{rc} implies $H_r^1\left(\bn\right)$ embeds compactly in $L^p(\bn)$. Thus there exists $w_{\varepsilon} \in H_r^1(\bn)$ such that, up to a subsequence,
\begin{equation}
u_n^\varepsilon \stackrel{n \rightarrow \infty}{\longrightarrow} w_{\varepsilon}\left\{\begin{array}{l}
\text { strongly in } L^p\left(\bn \right) \text{ for } 2<p<2^*  \\
\text { weakly in } H^1\left(\bn \right) \text { and in } L^{2^*}\left(\bn \right).
\end{array}\right. \label{cgt}
\end{equation}
Moreover, $w_{\varepsilon} \not\equiv 0$ follows from \eqref{neq}. 
Ekeland's variational principle (see section 3 of \cite{IE}) allows choosing the minimizing sequence $\left\{u_n^\varepsilon\right\}_n$ in $\mathcal{N}_{r}^{\varepsilon}$ such that
\begin{equation}
E_{\varepsilon, \infty}^{\prime}\left(u_n^\varepsilon\right)[v]=\lambda_n G^{\prime}\left(u_n^\varepsilon\right)[v]+o(1)\|v\| \quad \forall v \in H_r^1\left(\bn\right). \label{a.4}
\end{equation}
where, for all $n \in \mathbb{N}, \lambda_n \in \mathbb{R}$ is the Lagrange multiplier and $G(u)=E_{\varepsilon, \infty}^{\prime}(u)[u]$. Thus $G\left(u_n^\varepsilon\right)=0$ for all $n \in \mathbb{N}$, and \eqref{a.4} implies
\begin{equation*}
0=G\left(u_n^\varepsilon\right)=E_{\varepsilon, \infty}^{\prime}\left(u_n^\varepsilon\right)\left[u_n^\varepsilon\right]=\lambda_n G^{\prime}\left(u_n^\varepsilon\right)\left[u_n^\varepsilon\right]+o(1)\left\|u_n^\varepsilon\right\|_{\la}.\end{equation*}
Hence, we get $\lambda_n=o(1)$ using $\left\|u_n^\varepsilon\right\|_\la$ is bounded and $G^{\prime}\left(u_n^\varepsilon\right)\left[u_n^\varepsilon\right] \leq c<0$ on $\mathcal{N}_r^\varepsilon.$ Choosing $v=w_{\varepsilon}$ in \eqref{a.4}, by \eqref{cgt}, we can deduce $w_{\varepsilon} \in \mathcal{N}_{\varepsilon, \infty}$.\\
By utilising once more, we obtain
\begin{equation*}
m_{\varepsilon} \leq E_{\varepsilon, \infty}\left(w_{\varepsilon}\right) \leq \liminf _{n \rightarrow \infty}\left[\left(\frac{1}{2}-\frac{1}{2^*}\right)\left\|u_n^\varepsilon\right\|_{\la}^2-\left(\frac{1}{p}-\frac{1}{2^*}\right)\left|u_n^\varepsilon\right|_p^p\right]=m_{\varepsilon}.
\end{equation*}
This means that the minimizing function we are seeking for is $w_\varepsilon$. Consequently, $w_{\varepsilon}$ solves
\begin{equation*}
-\Delta_{\bn} u-\la u=|u|^{p-2} u+\varepsilon|u|^{2^*-2} u \quad \text { in } \bn \text {. }
\end{equation*}
Furthermore, as a result of the maximum principle, $w_\varepsilon$ is strictly positive.
\end{proof}

\medskip


\subsection{Blow-up argument}
In this section, we shall find a ground-state solution to our aimed problem, i.e., \eqref{mainprob} under the assumption $\left(\mathbf{A}_{1}\right).$  We prove the existence of a least energy positive solution $\tilde{u}$ of \eqref{mainprob}, i.e., $\tilde{u} \in \mathcal{N}_{\varepsilon}$ satisfies $E_{\varepsilon}(\tilde{u})= \inf _{\mathcal{N}_{\varepsilon}} E_{\varepsilon} $ where $\mathcal{N}_{\varepsilon}$ denotes the Nehari manifold corresponding to \eqref{mainprob}. The following auxiliary results are required to prove Theorem \ref{grst}. We first begin with the blow-up argument, which controls $m_{\varepsilon}.$ 
 
\begin{prop}\label{propres}
    The estimate mentioned  below holds:
\begin{equation}
m_{\varepsilon} \leq \frac{1}{N} S^{N / 2}\left(\frac{1}{\varepsilon}\right)^{\frac{N-2}{2}} \quad \forall \varepsilon>0 \label{gro}
\end{equation}
where $S$ is the best Sobolev constant that occurs in the Sobolev inequality in $\rn$.
\end{prop}
\begin{proof}

The energy functional corresponding to \eqref{conprob} is given by
\begin{align}
J_\varepsilon(v)&=\frac{1}{2} \int_{\mathrm{B}^N}\left[|\nabla v|^2-\left(\lambda-\frac{N(N-2)}{4}\right)\left(\frac{2}{1-|x|^2}\right)^2 v^2\right] \mathrm{d} x\\
&-\frac{1}{p}  \int_{\mathrm{B}^N}\left(\frac{2}{1-|x|^2}\right)^{N-\left(\frac{N-2}{2}\right)p}|v|^{p} \mathrm{d} x - \frac{\varepsilon}{2^*}\int_{\mathrm{B}^N}|v|^{2^*} \mathrm{d} x.
\end{align}
Then for any $u \in H^1\left(\mathbb{B}^N\right)$ if $\tilde{u}$ is defined as $\tilde{u}(x)=\left(\frac{2}{1-|x|^2}\right)^{\frac{N-2}{2}} u(x)$ then $E_{\varepsilon}(u)=$ $J_\varepsilon(v)$. Moreover, $\left\langle E_{\varepsilon}^{\prime}(u), v\right\rangle=\left\langle J_\varepsilon^{\prime}(\tilde{u}), \tilde{v}\right\rangle$ where $\tilde{v}$ is defined in the same way as $\tilde{u}$. \\
Consequently, we obtain 
\begin{equation*}
m_{\varepsilon}:=\inf _{\mathcal{N}_{\varepsilon, \infty}} E_{\varepsilon, \infty}= \inf _{\mathcal{N}_{J}^\varepsilon} J_{\varepsilon}
\end{equation*}
where $\mathcal{N}_J^\varepsilon$ denotes the Nehari manifold associated with the functional $J_{\varepsilon}$.\\
Now to prove \eqref{gro}, we exhibit a sequence $v_{n}$ in $\mathcal{N}_{J}^\varepsilon$ such that $J_{\varepsilon}\left(v_n\right)$\;converges to $\frac{1}{N} S^{N / 2}\left(\frac{1}{\varepsilon}\right)^{\frac{N-2}{2}}.$\\
Before moving further, observe that the value $\frac{1}{N} S^{N / 2}\left(\frac{1}{\varepsilon}\right)^{\frac{N-2}{2}}$ can be characterised as 
\begin{equation}
\begin{aligned}
\frac{1}{N} S^{N / 2}\left(\frac{1}{\varepsilon}\right)^{\frac{N-2}{2}}= \min  \left\{\right.&\frac{1}{2} \int_{\mathbb{R}^N}|\nabla u|^2 \mathrm{d} x-\frac{\varepsilon}{2^*} \int_{\mathbb{R}^N}|u|^{2^*} \mathrm{d} x: u \in \mathcal{D}^{1,2}\left(\mathbb{R}^N\right), \\
& \left.\int_{\mathbb{R}^N}|\nabla u|^2 \mathrm{d} x=\varepsilon \int_{\mathbb{R}^N}|u|^{2^*} \mathrm{d} x\right\}.
\end{aligned} \label{AB}
\end{equation}
We define the following sequence by scaling the Aubin-Talenti bubble:
\begin{equation*}
U_{n}(x):=n^{\frac{N-2}{2}} U(nx)=(N(N-2))^{\frac{N-2}{4}} \frac{n^{\frac{N-2}{2}}}{\left(1+|nx|^2\right)^{\frac{N-2}{2}}}, \quad x \in \mathbb{R}^N
\end{equation*}
where $U \in \mathcal{D}^{1,2}\left(\mathbb{R}^N\right)$ is a fixed radial function that realizes the minimum in \eqref{AB}.\\
Let $\eta \in C_c^{\infty}\left(\mathrm{B}^N\right)$ be a cut-off function such that $\eta(x) \equiv 1$ in a neighborhood of 0. Then consider $u_{n} \in D^{1,2}(\mathbb{R}^N)$  to be a sequence of test functions  defined by $u_{n}(x):=\eta(x) U_{n}(x)$ for $x \in \mathrm{B}^N$. Further, define a sequence of functions $v_n:=t_n u_n, n \in \mathbb{N}$, where $t_n$ is such that $v_n \in \mathcal{N}_{J}^\varepsilon$, that is
\begin{equation}
\begin{aligned}
&\int_{\mathrm{B}^N}\left[|\nabla u_n|^2-\left(\lambda-\frac{N(N-2)}{4}\right)\left(\frac{2}{1-|x|^2}\right)^2 u_n^2\right] \mathrm{d} x \\
&=t_n^{p-2}\int_{\mathrm{B}^N}\left(\frac{2}{1-|x|^2}\right)^{N-\left(\frac{N-2}{2}\right)p}|u_{n}|^{p} \mathrm{d} x +\varepsilon t_n^{2^*-2}\int_{\mathrm{B}^N}|u_n|^{2^*} \mathrm{d} x.
\end{aligned} \label{a5}
\end{equation}
Now we further compute each of the above terms separately and let $n \rightarrow \infty.$

\begin{align}\label{a1}
\int_{\mathrm{B}^N} u_{n}^{2^{\star}} \mathrm{d} x &=\int_{B(0,\frac{1}{2})} u_{n}^{2^{\star}} \mathrm{d}x + \int_{\mathrm{B}^N \setminus B(0,\frac{1}{2})}u_{n}^{2^{\star}} \mathrm{d}x=\int_{B(0,\frac{1}{2}n)}U^{2^{\star}}~\mathrm{d} x +\int_{B(0,\frac{1}{2}n)^c} \eta\left(\frac{x}{n}\right)^{2^{\star}} U^{2^{\star}}\mathrm{d} x \nonumber\\
&=\int_{\mathbb{R}^N} U^{2^{\star}} \mathrm{d} x+ O\left((1/n)^{N}\right). 
\end{align}

\begin{align}
 \int_{\mathrm{B}^N} \left(\frac{1-|x|^2}{2}\right)^{\left(\frac{N-2}{2}\right)p- N} u_{n}^p & \mathrm{~d} x = \left(\int_{B(0,\frac{1}{2})} + \int_{\mathrm{B}^N \setminus B(0,\frac{1}{2})}  \right) \left(\frac{2}{1-|x|^2}\right)^{N-\left(\frac{N-2}{2}\right)p} u_{n}^p \mathrm{~d} x \notag \\
 &= \left(\frac{1}{n}\right)^{N- \left(\frac{N-2}{2}\right)p} \int_{B(0,\frac{1}{2}n)} \left(\frac{2}{1-|\frac{x}{n}|^2}\right)^{N-\left(\frac{N-2}{2}\right)p} U^{p}\mathrm{d} x \notag \\
  &+ \left(\frac{1}{n}\right)^{N- \left(\frac{N-2}{2}\right)p} \int_{B(0,\frac{1}{2}n)^c} \left(\frac{2}{1-|\frac{x}{n}|^2}\right)^{N-\left(\frac{N-2}{2}\right)p}\eta\left(\frac{x}{n}\right)^{p} U^{p}\mathrm{d} x \notag \\
 &= \left(\frac{1}{n}\right)^{N- \left(\frac{N-2}{2}\right)p} \int_{\mathbb{R}^N} U^{p} \mathrm{d} x + O\left(\left(\frac{1}{n}\right)^{\left(\frac{N-2}{2}\right)p}\right).
 \end{align} \label{a2}
Similarly, for the gradient term, we have the following estimate
\begin{align}
\int_{\mathrm{B}^N} |\nabla u_{n}|^{2} \mathrm{d} x&= \left(\int_{B(0,\frac{1}{2})} + \int_{\mathrm{B}^N \setminus B(0,\frac{1}{2})}  \right)|\nabla u_{n}|^{2} \mathrm{d}x\nonumber\\&=\int_{B(0,\frac{1}{2}n)}|\nabla U|^{2}~\mathrm{d} x + O\left((1/n)^{N-2}\right)\nonumber\\
&=\int_{\mathbb{R}^N}|\nabla U|^2 \mathrm{~d} x+O\left((1/n)^{N-2}\right). \label{a3}
\end{align}
Further, for the lower-order terms, we can obtain the following as $n \rightarrow \infty$
\begin{equation}
\begin{aligned}
\int_{\mathrm{B}^N} \left[\frac{4 \lambda-N(N-2)}{\left(1-|x|^2\right)^2}\right] u_{n}^2 \mathrm{~d} x =
 \left \{ \begin{array} {lc}
           \left(4\lambda-N(N-2)\right)\left(\frac{1}{n}\right)^{2}  \left[~\displaystyle\int_{\mathbb{R}^{n}} U^{2} \mathrm{~d} x  + o(1)\right] \quad    &\hbox{ for }~ N\geq 5, \\
           \left(4\lambda-N(N-2)\right)\left(\frac{1}{n}\right)^{2}   \log(n) \left[\omega_{N-1} + o(1) \right] \quad  &\hbox{ for  }~N =4.
            \end{array} \right.
\end{aligned} \label{a4}
\end{equation}

Hence, taking $n \rightarrow \infty$ in \eqref{a5}, applying the estimates \eqref{a1} - \eqref{a4} and using the fact that $U$ realizes the minimum in \eqref{AB}, $t_n \rightarrow 1$ follows and so
\begin{equation*}
J_{\varepsilon}\left(v_n\right)-\left(\frac{1}{2} \int_{\mathbb{R}^N}|\nabla U|^2 \mathrm{~d} x-\frac{\varepsilon}{2^*} \int_{\mathbb{R}^N} U^{2^{\star}} \mathrm{d} x \right) \stackrel{n \rightarrow \infty}{\longrightarrow} 0 .
\end{equation*}
Hence \eqref{gro} follows since
\begin{equation*}
\begin{aligned}
& \frac{1}{2} \int_{\mathbb{R}^N}|\nabla U|^2 \mathrm{~d} x-\frac{\varepsilon}{2^*} \int_{\mathbb{R}^N} U^{2^{\star}} \mathrm{d} x 
=\frac{1}{N}\left(\frac{1}{\varepsilon}\right)^{\frac{N-2}{2}} S^{N / 2}.
\end{aligned}
\end{equation*}
\end{proof}

\noindent
In fact, it can be deduced that:
\begin{equation}
m_{\varepsilon}<\frac{1}{N} S^{N / 2}\left(\frac{1}{\varepsilon}\right)^{\frac{N-2}{2}} \text {for $\varepsilon$ small. } \label{1rc}
\end{equation}
For this, observe that by conformally transforming the sequence found in Proposition \eqref{propres} to $\bn$, we can exhibit a sequence $\left\{\hat{u}_n\right\}$ radial functions in $\mathcal{N}_{\varepsilon, \infty}$ that converges weakly to 0 in $L^{2^*}(\bn)$, and such that $E_{\varepsilon, \infty}\left(\hat{u}_n\right) \rightarrow \frac{1}{N}\left(\frac{1}{\varepsilon}\right)^{\frac{N-2}{2}} S^{N / 2}$.
On the other hand, we have already shown in the proof of Theorem \eqref{limthm} that, up to a subsequence, any minimising sequence of radial functions weakly converges to a nonzero minimising function of $E_{\varepsilon, \infty}$ on $\mathcal{N}_{\varepsilon, \infty}$ for small $\varepsilon$. Therefore, \eqref{1rc} follows.\\
The following lemma provides another, more accurate estimate of $m_{\varepsilon}$ for small $\varepsilon$ and examines its asymptotic behaviour. 

\begin{lem} \label{asym}
The relation $m_\varepsilon \leq m$ is true for any $\varepsilon>0$. $m_{\varepsilon} \leq m$. Moreover, the following holds
\begin{equation*}
\lim _{\varepsilon \rightarrow 0} m_{\varepsilon}=m .
\end{equation*}
\begin{proof} We have already established the inequality $m_{\varepsilon} \leq m$ in \eqref{a1.1}.
Now, for $\varepsilon \in\left(0, \varepsilon_0\right)$ assume $w_{\varepsilon}$ be the minimizing function found in Theorem \ref{limthm} and $t_{\varepsilon}>0$ be such that $t_{\varepsilon} w_{\varepsilon} \in \mathcal{N}_{\infty}$, i.e., 
\begin{equation}
t_{\varepsilon}=\left(\frac{\left\|w_{\varepsilon}\right\|_{\la}^2}{\left|w_{\varepsilon}\right|_p^p}\right)^{\frac{1}{p-2}} \text {. } \label{aac}
\end{equation}
Besides, we have the following expression for $E_{\varepsilon, \infty}\left(w_{\varepsilon}\right)$
\begin{equation*}
E_{\varepsilon, \infty}\left(w_{\varepsilon}\right)=\left(\frac{1}{2}-\frac{1}{p}\right)\left\|w_{\varepsilon}\right\|_\lambda^2+\left(\frac{1}{p}-\frac{1}{2^*}\right) \varepsilon\left|w_{\varepsilon}\right|_{2^*}^{2^*}=m_{\varepsilon} \leq m 
\end{equation*}
implying $\left\|w_{\varepsilon}\right\|_{\la}$ is bounded, uniformly with respect to $\varepsilon \in\left(0, \varepsilon_0\right)$.
Moreover, $\left|w_{\varepsilon}\right|_p^p \geq c>0$ holds using \eqref{neq}. Consequently, \eqref{aac} suggests that $t_{\varepsilon}$ is bounded and
\begin{equation*}
\begin{aligned}
m & \leq E_{\infty}\left(t_{\varepsilon} w_{\varepsilon}\right)=E_{\varepsilon, \infty}\left(t_{\varepsilon} w_{\varepsilon}\right)+\frac{\varepsilon}{2^*} \int_{\bn}\left(t_{\varepsilon} w_{\varepsilon}\right)^{2^*} \mathrm{~d} V_{\bn}\\
& \leq E_{\varepsilon, \infty}\left(w_{\varepsilon}\right)+\frac{\varepsilon}{2^*} \int_{\bn}\left(t_{\varepsilon} w_{\varepsilon}\right)^{2^*} \mathrm{~d} V_{\bn}\\
& =m_{\varepsilon}+o(1) .
\end{aligned}
\end{equation*}
With this, we have completed our proof.
\end{proof}
\end{lem}


\medskip

In the subsequent proposition, we analyse the Palais-Smale sequences at a level $c\left((\mathrm{PS})_c\right.$-sequences) that will be 
beneficial in establishing the existence of a ground-state solution. We adapt the approach as described in \cite[Proposition~3.2]{SR} in our setting. 


\medskip

\begin{prop} \label{propa1}
    Let $a(x) \in C\left(\mathbb{B}^{N}\right) \cap L^{\infty}(\bn)$ satisfies $\left(\mathbf{A}_{1}\right).$ Let $\varepsilon>0$ and $\left\{u_n\right\}_n$ be a $(P S)_c$-sequence for $E_{\varepsilon}$ constrained on $\mathcal{N}_{\varepsilon}$. If $c< m_{\varepsilon},$ \text{then} $\left\{u_n\right\}_n$ is relatively compact.
\end{prop}
\begin{proof}
The proof is divided into several steps. 
\quad
	\begin{enumerate}[label = \textbf{Step \arabic*:}]
		\item Let us commence by noticing that the sequence $\left\{\left\|u_n\right\|_\la \right\}_n$ is bounded above since it is a PS sequence, and is bounded away from 0 as it belongs to $\mathcal{N}_{\varepsilon}$. 
Then, using the same approach as in Theorem \ref{limthm}, we can show that $\left\{u_n\right\}_n$ is a $(P S)_c$-sequence for the functional $E_{\varepsilon}$ as well, that is, $\forall v \in H^1(\bn)$
\begin{align}
   & \int_{\bn} \nabla_{\bn} u_n \cdot \nabla_{\bn} v \mathrm{~d} V_{\bn}- \la \int_{\bn}  u_n v \mathrm{~d} V_{\mathbb{B}^{N}}- \int_{\bn} a(x)\left|u_n\right|^{p-2} u_n v \mathrm{~d} V_{\mathbb{B}^{N}} \notag \\
   &-\varepsilon \int_{\bn}\left|u_n\right|^{2^*-2} u_n v \mathrm{~d} V_{\mathbb{B}^{N}}  
    =o(1)\|v\|_\lambda, \label{1a}
    \end{align}
    and
    $E_{\varepsilon}(u_n)\rightarrow c$ as $n \rightarrow \infty$.\\
In the subsequent steps, $\left\{u_n\right\}_n$ will denote the sequence and its subsequences.\\
The boundedness of $\left\{u_n\right\}_n$ in $H^1(\bn)$ implies existence of a function $\bar{u} \in H^1(\bn)$ such that
\begin{equation}
u_n \stackrel{n \rightarrow \infty}{\longrightarrow} \bar{u} \quad\left\{\begin{array}{l}
\text { weakly in } H^1(\bn) \text { and in } L^{2^*}(\bn) \\
\text { strongly in } L_{\mathrm{loc}}^p(\bn) \text { and in } L_{\mathrm{loc}}^2(\bn) \\
\text { a.e. in } \bn.
\end{array}\right. \label{1b}
\end{equation}
Then using \eqref{1b} and \eqref{1a}, we get $\bar{u}$ is a weak solution of \eqref{mainprob}, hence
\begin{equation}
\|\bar{u}\|_\la=|a^{1/p}\bar{u}|_p^p+\varepsilon|\bar{u}|_{2^*}^{2^*}. \label{1c}
\end{equation}
\item We intend to prove that $u_n \rightarrow \bar{u}$ in $H^1(\bn)$.\\
We will argue by contradiction that if $u_n \nrightarrow \bar{u}$ in $H^1(\bn)$, then the sequence $v_n:=u_n-\bar{u}$ verifies $\left\|v_n\right\|_\la \geq \hat{c}>0, \forall n \in \mathbb{N}$. Utilising \eqref{1b} and the Brezis-Lieb Lemma,
\begin{equation}
E_{\varepsilon}\left(u_n\right)=E_{\varepsilon}(\bar{u})+E_{\varepsilon}\left(v_n\right)+o(1). \label{1l}
\end{equation}
Additionally, because $\bar{u}$ is a solution of \eqref{mainprob}$,\left\{v_n\right\}_n$ also becomes a (PS)-sequence for $E_{\varepsilon}$.
We now claim the following
\begin{equation}
\left|v_n\right|_{2^*}^{2^*} \geq \tilde{c}>0. \label{1d}
\end{equation}
If not, $u_n \rightarrow \bar{u}$ in $L^{2^*}(\bn)$ and via interpolation in $L^p(\bn)$, since $\left\{u_n\right\}_n$ is bounded in $L^2(\bn)$.
 Further, by $\left\|u_n\right\|_\la=\left|a^{1/p}u_n\right|_p^p+\varepsilon\left|u_n\right|_{2^*}^{2^*}$ and \eqref{1c}, we have
\begin{equation*}
\lim _{n \rightarrow \infty}\left\|u_n\right\|_\la=|a^{1/p}\bar{u}|_p^p+\varepsilon|\bar{u}|_{2^*}^{2^*}=\|\bar{u}\|_\la
\end{equation*}
implying  $u_n \rightarrow \bar{u}$ in $H^1(\bn)$, leading to a contradiction.
\item Now covering $\mathbb{B}^{N}$ by balls of radius $r$, in such a way that each point of $\mathbb{B}^{N}$ is contained inside at most $N_0$ balls. The fact that $v_n \in L^{2^*}\left(\bn\right)$ allows us to define
\begin{equation*}
d_n=\sup _{y \in \bn} |v_{n}|_{L^{2^*}(B_r(y))} \quad \forall n \in \mathbb{N} .
\end{equation*}
Using \eqref{1d} and the boundedness of $\left\{u_n\right\}_n$ in $H^1(\bn)$, we get
\begin{equation}
\begin{aligned}
0<\tilde{c} \leq\left|v_n\right|_{2^*}^{2^*} & =\sum_{y \in \bn}^{}\left|v_n\right|_{L^{2^*}\left(B_r(y)\right)}^{2^*} \\
& \leq d_n^{2^*-2} \sum_{y \in \bn}^{}\left|v_n\right|_{L^{2^*}\left(B_r(y) \right)}^{2} \leq c d_n^{2^*-2} \sum_{y \in \bn}^{}\left\|v_n\right\|_{H^1\left(B_r(y)\right)}^2 \\
& \leq c^{\prime} d_n^{2^*-2}.
\end{aligned} \label{1i}
\end{equation}
Thus $d_n \geq \gamma \;\;\forall n \in \mathbb{N}$, where $\gamma>0$.
Consequently, we can find a sequence $\left\{z_{n}\right\} \subset \mathbb{B}^{N}$ such that

\begin{equation}
	\int_{B_r(z_{n})}\left|v_{n}\right|^{2^*}\mathrm{~d} V_{\mathbb{B}^{N}}\geq \gamma >0. \label{2n}
\end{equation}
Further, define
\begin{equation*}
w_{n}(x):=v_{n}\left(\tau_{z_{n}}(x)\right)
\end{equation*}
where $\tau_{z_{n}}$ is the hyperbolic translation of $\mathbb{B}^{N}$ by $z_{n}$.
Since the sequence $\left\{w_n\right\}_n$ is bounded in $H^1(\bn)$,  there exists some $\bar{w} \in H^1\left(\bn\right)$ such that
\begin{equation}
w_n \stackrel{n \rightarrow \infty}{\longrightarrow} \bar{w} \quad\left\{\begin{array}{l}
\text { weakly in } H^1\left(\bn\right) \text { and in } L^{2^*}\left(\bn\right) \\
\text { strongly in } L_{\text{loc}}^p\left(\bn\right) \text { and in } L_{\mathrm{loc}}^2\left(\bn\right) \\
\text { a.e. in } \bn .
\end{array}\right. \label{1g}
\end{equation}
Setting $L_0= B_1(0)$, one of the following two scenarios can happen:\\
\begin{equation*}
(a) \int_{L_0}\left|w_n(x)\right|^p \mathrm{~d} V_{\mathbb{B}^{N}}(x) \geq c>0 
\end{equation*}
\begin{equation}
(b) \int_{L_0}\left|w_n(x)\right|^p \mathrm{~d} V_{\mathbb{B}^{N}}(x) \stackrel{n \rightarrow \infty}{\longrightarrow} 0. \label{1e}
\end{equation}
\item First of all, let's assume that \eqref{1e}$(a)$ is true. Then, as $v_n \rightarrow 0$ in $L_{\text{loc}}^p\left(\bn\right)$ will imply $\tau_{z_{n}}(0) \rightarrow \infty \text { as } n \rightarrow \infty$. Further, combining the facts that $\left\{u_n\right\}_n$ is a $(P S)$-sequence and $\bar{u}$ is a solution of \eqref{mainprob}, we can obtain the following
\begin{equation}
\begin{aligned}
\hspace{1.1cm}&\int_{\bn} \nabla_{\bn} w_n  \cdot \nabla_{\bn} \phi \mathrm{~d} V_{\mathbb{B}^{N}}- \la \int_{\bn} w_n \phi \mathrm{~d} V_{\mathbb{B}^{N}}-\int_{\bn}\left|w_n\right|^{p-2} w_n \phi \mathrm{~d} V_{\mathbb{B}^{N}}\\
&\hspace{9cm}-\varepsilon \int_{\bn}\left|w_n\right|^{2^*-2} w_n \phi \mathrm{~d} V_{\mathbb{B}^{N}} \\
& =\int_{\bn}\left[a\left(\tau_{z_{n}}(\cdot)\right)-1\right] \left|w_n\right|^{p-2} w_n \phi\mathrm{~d} V_{\mathbb{B}^{N}}+o(1)\|\phi\|=o(1)\|\phi\|, \quad \forall \phi \in \mathcal{C}_0^{\infty}\left(\bn\right).&
\end{aligned} \label{1f}
\end{equation}
\noi
As a result of \eqref{1e}$(a)$, \eqref{1f}, \eqref{1g}, we are able to conclude that $\bar{w}$ is a nonzero solution of $\eqref{limprob}$. Then $\left\{w_n-\bar{w}\right\}_n$ is a $(P S)$-sequence for $E_{\varepsilon, \infty}$ and $E_{\varepsilon, \infty}\left(w_n-\bar{w}\right) \geq o(1)$. Consequently, by using the Brezis-Lieb Lemma, we derive
\begin{equation*}
\begin{aligned}
c & =E_{\varepsilon}\left(u_n\right)+o(1)=E_{\varepsilon}(\bar{u})+E_{\varepsilon, \infty}(\bar{w}) \\
& +E_{\varepsilon, \infty}\left(w_n-\bar{w}\right)+o(1) \geq E_{\varepsilon, \infty}(\bar{w})+o(1) \geq m_{\varepsilon}+o(1)
\end{aligned}
\end{equation*}
that contradicts the assumption that $c<m_{\varepsilon}$. Thus we can conclude that \eqref{1e}$(a)$ can not hold.

\item Finally, we suppose that \eqref{1e}$(b)$ holds and get a contradiction again. Notice that in this case, we can also assume the following
\begin{equation}
\tilde{d}_n=\sup _{y \in \bn} |v_{n}|_{L^{p}(B_r(y))}=\sup _{y \in \bn} |w_{n}|_{L^{p}(B_r(y))}\stackrel{n \rightarrow \infty}{\longrightarrow} 0. \label{1h}
\end{equation}
Indeed, if it is false, we can land in the case \eqref{1e}$(a)$ by replacing $B_r(y)$ with a ball $B_r(\tilde{y})$ that satisfies $\left|v_n\right|_{L^p\left(B_r(\tilde{y})\right)} \geq c_1>0$. So, it is possible to assume \eqref{1h}. Furthermore, repeating the inequalities in \eqref{1i} with the $L^p$-norm instead of the $L^{2^*}$-norm, and $\tilde{d}_n$ instead of $d_n$ yield
\begin{equation}
\left|v_n\right|_p=\left|w_n\right|_p \stackrel{n \rightarrow \infty}{\longrightarrow} 0 . \label{1j}
\end{equation}
Thus implying $\bar{w} =0$. 
Also, observe that \eqref{1j} gives
\begin{equation}
w_n \stackrel{n \rightarrow \infty}{\longrightarrow} 0 \;\; \text { in } L_{\text{loc}}^q(\bn) \text{ for $q<p$}. \label{1k}
\end{equation}
Now let's consider that $\left\{z_n\right\}_n$ is bounded so that in the subsequent argument we can consider $z_n=0, \forall n \in \mathbb{N}$. Further, fix $R>0$ such that $|a(x)-1|<\eta\; \forall x \in \mathbb{B}^N \setminus B_R(0)$, where $\eta$ is a suitable small constant that will be chosen later. \\
Consider the functionals $I_{\varepsilon}$: $\mathcal{D}^{1,2}\left(\mathbb{B}^N\right) \rightarrow \mathbb{R}$ and $J_{\epsilon, \infty}$: $\mathcal{D}^{1,2}\left(\mathbb{R}^N\right) \rightarrow \mathbb{R}$ defined by
\begin{equation*}
\begin{gathered}
I_{\varepsilon}(u)=\frac{1}{2} \int_{\bn}\left(|\nabla_{\bn} u|^2- \la u^2\right) \mathrm{~d} V_{\bn}- \frac{\varepsilon}{2^*} \int_{\bn}|u|^{2^*} \mathrm{~d} V_{\bn}, \\
J_{\varepsilon, \infty}(u)=\frac{1}{2} \int_{\mathbb{R}^N}|\nabla u|^2 \mathrm{d} x-\frac{\varepsilon}{2^*} \int_{\mathbb{R}^N}|u|^{2^*} \mathrm{d} x .
\end{gathered}
\end{equation*}
Firstly, observe that the Euler Lagrange equation (say $(AP_{\infty})$) corresponding to the functional $I_{\varepsilon}$ is invariant under hyperbolic isometries. Thus for a solution $U$ of $(AP_{\infty})$,  if we define
\begin{equation}
u_n=U \circ \tau_n \label{ps1}
\end{equation}
where $\tau_n \in$ $I\left(\mathbb{B}^N\right)$ with $\tau_n(0) \rightarrow \infty$ and $I\left(\mathbb{B}^N\right)$ denotes the isometry group of $\bn$, then $\{u_n\}_n$ is a PS sequence converging weakly to zero. Moreover, in the case of critical exponent, another PS sequence can be exhibited emerging from the concentration phenomenon as follows (\cite{BS}). \\
Let $V$ be a solution of the Euler-Lagrange equation corresponding to the functional $J_{\epsilon, \infty}$. Fix $x_0 \in \mathbb{B}^N$ and $\phi \in C_c^{\infty}\left(\mathbb{B}^N\right)$ such that $0 \leq \phi \leq 1$ and $\phi \equiv 1$ in a neighborhood of $x_0$. Define
\begin{equation}
v_n(x)=\left(\frac{1-|x|^2}{2}\right)^{\frac{N-2}{2}} \phi(x) \varepsilon_n^{\frac{2-N}{2}} V\left(\left(x-x_0\right) / \varepsilon_n\right) \label{ps2}
\end{equation}
where $\varepsilon_n>0$ and $\varepsilon_n \rightarrow 0$, then $v_n$ is also a PS sequence. \\

Then using \eqref{1a}, \eqref{1j}, and \eqref{1k}, we can get that $\left\{w_n\right\}_n$ is a (PS)-sequence for $I_{\varepsilon}(u)$. So, recalling Theorem $3.3$ of \cite{BS}: there exist $n_1, n_2 \in \mathbb{N}$ and functions $u_n^j \in H^1(\mathbb{B}^N), 0 \leq j \leq n_1, v_n^k \in H^1(\mathbb{B}^N), 0 \leq k \leq n_2$ and $w \in H^1\left(\mathbb{B}^N\right)$ such that upto a subsequence
\begin{equation*}
w_n=w+\sum_{j=1}^{n_1} u_n^j+\sum_{k=1}^{n_2} v_n^k+o(1)
\end{equation*}
where $I^{\prime}_{\varepsilon}(w)=0, u_n^j, v_n^k$ are $P S$ sequences of the form \eqref{ps1} and \eqref{ps2}, respectively and $o(1) \rightarrow 0$ in $H^1(\mathbb{B}^N)$. Moreover
\begin{equation}
I_{\varepsilon}(w_{n})=I_{\varepsilon}(w)+\sum_{j=1}^{n_1} I_{\varepsilon}\left(U_j\right)+\sum_{k=1}^{n_2} J_{\varepsilon, \infty}\left(V_k\right)+o(1) \label{ps4}
\end{equation}
where $U_j, V_k$ are the solutions Euler Lagrange equation associated with $I_{\varepsilon}$ and $J_{\varepsilon, \infty}$, respectively, corresponding to $u_n^j$, and $v_n^k$.\\
Moreover, \eqref{AB} implies
\begin{equation}
J_{\varepsilon, \infty}\left(V_k\right) \geq \frac{1}{N} S^{N / 2}\left(\frac{1}{\varepsilon}\right)^{\frac{N-2}{2}} \label{1m}
\end{equation}
where $S$ is the best Sobolev constant that occurs in the Sobolev inequality in $\rn$.\\
Finally, by \eqref{1l},\eqref{1m},\eqref{1j},\eqref{ps4} and Proposition \ref{propres}, we have
\begin{equation*}
\begin{aligned}
E_{\varepsilon}\left(u_n\right) & =E_{\varepsilon}(\bar{u})+E_{\varepsilon}\left(v_n\right)+o(1) \\
& \geq E_{\varepsilon}(\bar{u})+I_{\varepsilon}\left(w_n\right)- \frac{1}{p} \int_{\mathbb{B}^N \setminus B_R(0)}(a(x)-1)|v_{n}|^p \mathrm{~d} V_{\bn}+o(1) \\
& \geq E_{\varepsilon}(\bar{u})+\sum_{k=1}^{n_2} J_{\varepsilon, \infty}\left(V_k\right) - \eta \hat{c}+o(1) \\
& \geq \frac{1}{N} S^{N / 2}\left(\frac{1}{\varepsilon}\right)^{\frac{N-2}{2}}- \eta \hat{c}+o(1) \\
& \geq m_{\varepsilon} -\eta \hat{c} + o(1)\\
& > c \\
\end{aligned}
\end{equation*}
for $n$  large and $\eta$ sufficiently small enough. Therefore, it contradicts our assumption $E_{\varepsilon}\left(u_n\right) \rightarrow c$. To conclude, let's take into account the case $d(z_n,0) \rightarrow \infty$. In this case, it is simpler to repeat the argument implemented in the case $\left\{z_n\right\}_n$ bounded.
\end{enumerate}
\end{proof}
\begin{proof}[Proof of Theorem \ref{thm1}]

The following inequality will be put to use for establishing a solution
\begin{equation}
\inf _{\mathcal{N}_{\varepsilon}} E_{\varepsilon}< m_{\varepsilon}. \label{ps5}
\end{equation}
In order to prove the above inequality, take into account the minimising function $w_\varepsilon$ provided by the Theorem \ref{limthm}, and $t>0$ be such that $t w_{\varepsilon} \in \mathcal{N}_{\varepsilon}$. Then
\begin{equation*}
\inf _{\mathcal{N}_{\varepsilon}} E_{\varepsilon} \leq E_{\varepsilon}\left(t w_{\varepsilon}\right)<E_{\varepsilon, \infty}\left(t w_{\varepsilon}\right) \leq E_{\varepsilon, \infty}\left(w_{\varepsilon}\right)=m_{\varepsilon}.
\end{equation*} 
It is implied by \eqref{ps5} and the Proposition \ref{propa1} that a minimising function $\tilde u$ exists for the functional $E_\varepsilon$ constrained on $\mathcal{N_\varepsilon}$. 
Further, one can verify that $\tilde u$ is a constant sign function that can be chosen strictly positive by employing the same argument used in the proof of Theorem \ref{limthm}.
\end{proof}

\section{ bound-state solution: Existence and Non-existence } \label{sec4}

This section is devoted to the proof of Theorem~\ref{aless} concerning the solutions of \eqref{mainprob} for $a(x) \leq 1$. Firstly, we show the non-existence of a ground-state solution. Then we restore (local) compactness in a range of functional values and prove Theorem~\ref{aless}. In particular, we prove the following proposition, which led us to a non-existence result:
\begin{prop}Let $\varepsilon \in\left[0, \varepsilon_0\right)$. Assume $a(x) \leq 1$, $a(x) \not \equiv 1$, then
\begin{equation}
\inf _{\mathcal{N}_{\varepsilon}} E_{\varepsilon}=m_{\varepsilon}.\label{4a}
\end{equation}
Furthermore, there is no solution to the minimization problem \eqref{4a}. Moreover, here $E_{0}:=E, \;\mathcal{N}_{0}:= \mathcal{N},\;m_0 :=m$. \label{prop2ab}
\end{prop}
\noi
First, note that when  $\varepsilon=0,$ \eqref{mainprob} becomes
\begin{equation*}
-\Delta_{\bn} u \, - \,  \lambda u \, = \, a(x) u^{p-1}  \text { in } \bn, \;\;u>0 \text { in } \bn,\;\;u \in H^1(\bn), \label{prob1} \tag{$P$}
\end{equation*}
and the corresponding energy functional is $E: H^1(\bn) \rightarrow \mathbb{R}$ defined by
\begin{equation*}
E(u)=\frac{1}{2} \int_{\bn}\left(|\nabla_{\bn} u|^2- \la u^2\right) \mathrm{~d} V_{\mathbb{B}^{N}}-\frac{1}{p} \int_{\bn} a(x)|u|^p \mathrm{~d} V_{\mathbb{B}^{N}},
\end{equation*}
and $\mathcal{N}$ denotes the corresponding Nehari Manifold.

\begin{proof}[Proof of Proposition \ref{prop2ab}]
    Let $u \in \mathcal{N}_{\varepsilon}$ and $t_u \in \mathbb{R}$ be such that $t_u u \in \mathcal{N}_{\varepsilon, \infty}$. Using the assumption that $a(x) \leq 1$ a.e. in $\mathbb{B}^N$, we have
\begin{equation*}
m_{\varepsilon} \leq E_{\varepsilon, \infty}\left(t_u u\right) \leq E_{\varepsilon}\left(t_u u\right) \leq E_{\varepsilon}(u),
\end{equation*}
which in turn implies, $\inf_{\mathcal{N}_{\varepsilon}} E_{\varepsilon} \geq m_{\varepsilon}$. Now we shall exhibit a sequence $\{ v_n \}_n \in 
\mathcal{N}_{\varepsilon}$ such that $E_{\varepsilon}\left(v_n\right) \rightarrow m_{\varepsilon}$ as $n \rightarrow \infty.$ This will prove our desired result. 


\medskip

 Define $v_{n}(x)=t_{n}w_{\varepsilon}(\tau_{n}(x))$ where $\tau_n$ is the hyperbolic translation with $\tau_n(0)=b_n$, and $b_n \in \mathbb{B}^N$ such that $b_n \rightarrow \infty$,  $w_{\varepsilon}$ is the minimizing function (solution) established in Theorem~\ref{limthm} and $t_n>0$ is such that $v_{n}(x)=t_{n}w_{\varepsilon}(\tau_{n}(x)) \in \mathcal{N}_{\varepsilon}$.\\
Now $v_{n} \in \mathcal{N}_{\varepsilon}$ gives
\begin{equation}
    \left\|w_{\varepsilon}(\tau_{n}(\cdot))\right\|_{\la}^2-t_{n}^{p-2}|a ^\frac{1}{p}w_{\varepsilon}(\tau_{n}(\cdot))|_p^p-\varepsilon t_{n}^{2^*-2}\left|w_{\varepsilon}(\tau_{n}(\cdot))\right|_{2^*}^{2^*}=0, \label{2aa}
\end{equation}
that is
\begin{equation*}
t_{n}^{p-2}|a ^\frac{1}{p}w_{\varepsilon}(\tau_{n}(\cdot))|_p^p+\varepsilon t_{n}^{2^*-2}\left|w_{\varepsilon}\right|_{2^*}^{2^*} =\left\|w_{\varepsilon}\right\|_{\lambda}^2.
\end{equation*}
Consequently, $\left\{t_n\right\}_n$ is bounded and, up to a subsequence, $t_n\rightarrow t$. Letting $n\rightarrow \infty$ in \eqref{2aa} yields
\begin{equation*}
\left\|w_{\varepsilon}\right\|_\lambda^2-t^{p-2}\left|w_{\varepsilon}\right|_p^p-\varepsilon t^{2^*-2}\left|w_{\varepsilon}\right|_{2^*}^{2^*}=0,
\end{equation*}
giving us $t w_{\varepsilon} \in \mathcal{N}_{\varepsilon, \infty}$. But $w_{\varepsilon} \in \mathcal{N}_{\varepsilon, \infty}$, therefore, $t=1$. Then it is not difficult to see that $E_{\varepsilon}\left(v_n\right) \rightarrow$ $E_{\varepsilon, \infty}\left(w_{\varepsilon}\right)=m_{\varepsilon}$ and we can conclude $\inf _{\mathcal{N}_{\varepsilon}} E_{\varepsilon} \leq m_{\varepsilon}$.


\medskip

Furthermore, we want to show that $m_{\varepsilon}$ is not attained in $\mathcal{N}_{\varepsilon}$. We argue by contradiction, let $u \in \mathcal{N}_{\varepsilon}$ be such that $E_{\varepsilon}(u)=m_{\varepsilon}$.
Let $t>0$ be such that $t u \in \mathcal{N}_{\varepsilon, \infty}$, then
\begin{equation*}
m_{\varepsilon}=E_{\varepsilon}(u) \geq E_{\varepsilon}(t u)>E_{\varepsilon, \infty}(t u) \geq m_{\varepsilon},
\end{equation*}
which leads to a contradiction, and hence the proof is complete.
\end{proof}


\medskip

\subsection{Bound-state solutions}

Proposition~\ref{prop2ab} tells us that the problem~\eqref{mainprob} for $a(x) \leq 1$ does not admit a ground-state solution. This tempts us to look for high-energy solutions in the spirit of Bahri-Li. This idea has already been exploited for (purely) subcritical problems in the hyperbolic space, see e.g., \cite{DDS, DDS1}. It hinges on delicate interaction estimates of two hyperbolic bubbles. Let us first recall some of the well-known facts from the previous papers. 


\medskip

We recall the following standard compactness result in the subcritical case as a consequence of (\cite[Proposition~3.1]{DDS}) for $a(x)$ satisfying the conditions mentioned in Theorem \ref{aless}.
\begin{prop} \label{pscric}
    Assume $\left\{v_n\right\}_n$ to be a $(P S)_c$-sequence of $E$ for $ c \in (m, 2 m)$, then $\left\{v_n\right\}_n$ is relatively compact and, up to a subsequence, converges to a nonzero function $\bar{v} \in H^1(\bn)$ such that $E(\bar{v}) \in(m, 2 m)$.
\end{prop}

\medskip 

The above proposition led us to the following compactness result corresponding to the critical perturbation problem.

\begin{prop} \label{propps}
    For every $\delta \in(0, m / 2)$ there corresponds $\varepsilon_\delta>0$ satisfying the following property:  $\forall \varepsilon \in\left(0, \varepsilon_\delta\right), \forall c \in(m+\delta, 2 m-\delta)$, if $\left\{u_n\right\}_n$ is a $(P S)_{c} -$ sequence of $E_{\varepsilon}$ constrained on $\mathcal{N}_{\varepsilon}$, then $u_n \rightharpoonup \bar{u} \not\equiv 0$ weakly in $H^1(\bn)$. Furthermore, $\bar{u}$ is a critical point of $E_{\varepsilon}$ on $\mathcal{N}_{\varepsilon}$ and $E_{\varepsilon}(\bar{u}) \leq c$.
\end{prop}
\begin{proof}
    First of all, note that every $(\mathrm{PS})_c$ - sequence for the constrained functional is also a $(\mathrm{PS})_c$ - sequence for the free functional, and its weak limit is a critical point which can be seen by following similar arguments as in Proposition~\ref{propa1}. Also, (PS) sequence is bounded in $H^1(\bn)$, so it ensures the existence of a weak limit in $H^1\left(\bn \right)$. \\
    Now to prove the Proposition in hand, we argue by contradiction. We can assume that there exist $\bar{\delta} \in(0, m / 2)$, a sequence $\left\{c_n\right\}_n$ in $(m+\bar{\delta}, 2 m-\bar{\delta})$, a sequence $\left\{\varepsilon_n\right\}_n$ in $(0,+\infty)$, with $\varepsilon_n \rightarrow 0$, and, for every $n \in \mathbb{N}$, a sequence $\left\{u_k^n\right\}_k$ in $H^1(\bn)$ such that
\begin{equation*}
\begin{gathered}
E_{\varepsilon_n}\left(u_k^n\right) \stackrel{k \rightarrow \infty}{\longrightarrow} c_n, \quad E_{\varepsilon_n}^{\prime}\left(u_k^n\right) \stackrel{k \rightarrow \infty}{\longrightarrow} 0, \\
u_k^n \stackrel{k \rightarrow \infty}{\rightharpoonup} 0 \quad \text { weakly in } H^1(\bn) .
\end{gathered}
\end{equation*}
Additionally, as $p < 2^{\star},$ we can also assume 
\begin{equation*}
u_k^n \stackrel{k \rightarrow \infty}{\longrightarrow} 0 \quad \text { in } L_{\text {loc}}^p\left(\mathbb{B}^N\right) .
\end{equation*}
Up to a subsequence, $c_n \rightarrow \bar{c} \in[m+\bar{\delta}, 2 m-\bar{\delta}]$, and by using a diagonalization argument, we construct the sequence $\left\{v_n\right\}_n:=\left\{u_{k_n}^n\right\}_n$ that verifies
\begin{equation}
E_{\varepsilon_n}\left(v_n\right) \stackrel{n \rightarrow \infty}{\longrightarrow} \bar{c}, \quad E_{\varepsilon_n}^{\prime}\left(v_n\right) \stackrel{n \rightarrow \infty}{\longrightarrow} 0, \quad v_n \stackrel{n \rightarrow \infty}{\longrightarrow} 0 \quad \text { in } L_{\text {loc }}^p\left(\mathbb{B}^N\right) . \label{2b}
\end{equation}
Moreover, the following equality guarantees $\left\{\left\|v_n\right\|_\la\right\}_n$ is bounded
\begin{equation}
E_{\varepsilon_n}\left(v_n\right)=\left(\frac{1}{2}-\frac{1}{p}\right)\left\|v_n\right\|_\la^2+\left(\frac{1}{p}-\frac{1}{2^*}\right) \varepsilon_n\left|v_n\right|_{2^*}^{2^*}=\bar{c}+o(1). \label{2c}
\end{equation}
Thus we can deduce that
\begin{equation*}
\begin{gathered}
E\left(v_n\right)=E_{\varepsilon_n}\left(v_n\right)+\frac{\varepsilon_n}{2^*} \int_{\bn}\left|v_n\right|^{2^*} \mathrm{d} x \stackrel{n \rightarrow \infty}{\longrightarrow} \bar{c} \\
\left\|E^{\prime}\left(v_n\right)\right\|_{H^{-1}(\bn)} \leq\left\|E_{\varepsilon_n}^{\prime}\left(v_n\right)\right\|_{H^{-1}(\bn)}+C \varepsilon_n\left\|v_n\right\|_{\lambda}^{2^*-1} \stackrel{n \rightarrow \infty}{\longrightarrow} 0,
\end{gathered}
\end{equation*}
and the above steps imply $\left\{v_n\right\}_n$ is a $(P S)_{\bar{c}}$ - sequence of $E$, with $\bar{c} \in(m, 2 m)$. Then, according to Proposition \ref{pscric}$, \bar{v} \in H^1(\bn), \bar{v} \not\equiv 0$, exists such that $v_n \rightarrow \bar{v}$, in contrast to \eqref{2b}.

To conclude, if $\left\{u_n\right\}_n$ is a $(\mathrm{PS})_c$ - sequence for $E_{\varepsilon}$ constrained on $\mathcal{N}_{\varepsilon}$, and $u_n \rightharpoonup \bar{u}$, then replacing $\varepsilon_n$ with $\varepsilon$, and $\bar{c}$ with $c$ in \eqref{2c}, we can obtain $E_{\varepsilon}(\bar{u}) \leq c$.
\end{proof}

\medskip


\subsection{Energy Estimates}
In this subsection, we recall and establish some energy estimates for interacting hyperbolic bubbles. In addition, these functions are also used to analyse the sublevels of the functional $E_{\varepsilon}$. Subsequently, we construct barycenter-type maps to study a few properties of these sublevels.

\medskip

We first recollect the following energy estimate corresponding to $\eqref{prob1}$, purely subcritical problem (\cite[Lemma~4.2]{DDS}):

\begin{prop} \label{enerest}
Let $a(x)$ satisfy the conditions of Theorem \ref{aless}, and let $w$ be the unique radial solution of \eqref{infprob}. Then, there exists a large number $R_{0}$, such that for any $R \geq R_{0}$, and for any $x_{1}, x_{2}$ satisfying $$d(x_{1},0) \geq R^{\alpha},\; d(x_{2},0)\geq R^{\alpha},\; R^{\alpha^{\prime}} \leq d\left(x_{1},x_{2}\right) \leq R^{\alpha^{\prime}-\alpha} \min \left\{d(x_{1},0),d(x_{2},0)\right\},$$where $\alpha>\alpha^{\prime}>1$, it holds
\begin{equation}
J\left(t u_{1}+(1-t) u_{2}\right)<S_{2, \lambda}, \label{3h}
\end{equation}
where $0 \leq t \leq 1,$ and $ u_{i}=w\left(\tau_{-x_{i}}(\cdot)\right), i=1,2$,
$J, J_\infty: H^1(\bn)\rightarrow \mathbb R$ are defined as 
\begin{equation*}
J(u):=\frac{\|u\|_{\lambda}^{2}}{\left(\int_{\mathbb{B}^{N}} a(x)|u(x)|^{p} \mathrm{~d} V_{\mathbb{B}^{N}}(x)\right)^{\frac{2}{p}}}, J_{\infty}(u):=\frac{\|u\|_{\lambda}^{2}}{\left(\int_{\mathbb{B}^{N}}|u(x)|^{p} \mathrm{~d} V_{\mathbb{B}^{N}}(x)\right)^{\frac{2}{p}}},
\end{equation*}
and the energy levels as
\begin{equation*}
S_{1, \lambda}:=\inf _{u \in H^{1}(\mathbb{B}^{N}) \backslash\{0\}} J_{\infty}(u), \quad S_{2, \lambda}:= 2^{\frac{p-2}{p}} S_{1, \lambda}.
\end{equation*}
\end{prop}


\medskip

With the above proposition in mind, let us introduce some notations. 
Fix $x_{1} \in \bn$ such that $x_{1}$ satisfies the condition in Proposition \ref{enerest}. Precisely, choose $x_{1}$ such that $d(x_{1},0) = 2 R_{0}^\alpha$. Also, for all the notations as in the above proposition fix $R'=R_{0}^{\alpha^{\prime}}$.\\
\par
Now set $\Sigma=\partial B_{R^{'}}\left(x_1\right)$. Further, for any $R>0$ as in the above Proposition \ref{enerest}, define the map $\psi_R:[0,1] \times \Sigma \longrightarrow H^1(\bn)$ by
\begin{equation*}
\psi_R[s, y](x)=(1-s) w\left(\tau_{-x_{1}}(x)\right)+s w\left(\tau_{-y}(x)\right),
\end{equation*}
where $w$ is the ground-state solution of \eqref{infprob}. Further, let  $t_{R, s, y},t_{R, s, y}^{\prime}>0$ such that $    t_{R, s, y} \psi_R[s, y] \in \mathcal{N}_{\varepsilon}$ and $t_{R, s, y}^{\prime} \psi_R[s, y] \in \mathcal{N}$.\\
\par
We can deduce an energy estimate using Proposition \ref{enerest} in our case here.  In particular, we have the following proposition.
\begin{lem} \label{enerlem}
 There exists $\bar{R}>0$ and $\mathcal{A} \in(m, 2 m)$ such that for any $R>\bar{R}$ and for any $\varepsilon>0$
\begin{equation*}
\mathcal{A}_{\varepsilon, R}=\max \left\{E_{\varepsilon}\left(    t_{R, s, y} \psi_R[s, y]\right): s \in[0,1], y \in \Sigma\right\}<\mathcal{A}<2 m .
\end{equation*}
\end{lem}
\par
\begin{proof}
    For the sake of simplicity, we omit $s, y$ and write $t_R=t_{R, s, y}, t_R^{\prime}=t_{R, s, y}^{\prime}$ and $\psi_R=\psi_R[s, y]$. Using $    t_R^{\prime} \psi_R \in$ $\mathcal{N},$ we have
    \begin{equation*}
\left\|    t_R^{\prime} \psi_R\right\|_{\la}^2=\left|  a^{1/p}  t_R^{\prime} \psi_R\right|_p^p, \quad     t_R^{\prime}=\left(\frac{\left\|\psi_R\right\|_{\la}^2}{\left|a^{1/p}\psi_R\right|_p^p}\right)^{1 / p-2}
\end{equation*}
Now for every $\varepsilon>0$, we have
\begin{equation}
\begin{aligned}
E_{\varepsilon}\left(t_R \psi_R\right) & \leq E\left(t_R \psi_R\right) \leq E\left(    t_R^{\prime} \psi_R\right) \\ 
& =\frac{1}{2}\left\| t_R^{\prime} \psi_R\right\|_{\la}^2-\frac{1}{p}\left|a^{1/p}  t_R^{\prime} \psi_R\right|_p^p \\ 
& =\left(\frac{1}{2}-\frac{1}{p}\right)     \left(t_R^{\prime}\right)^{2}\left\|\psi_R\right\|_{\la}^2 \\ 
& =\left(\frac{1}{2}-\frac{1}{p}\right)\left(\frac{\left\|\psi_R\right\|_\la^2}{\left|a^{1/p}\psi_R\right|_p^2}\right)^{\frac{p}{p-2}}.
\end{aligned} \label{3ab}
\end{equation}
So to prove the lemma we are left to estimate the above ratio.
Note that 
$$\left(\frac{\left\|\psi_R\right\|_\la^2}{\left|a^{1/p}\psi_R\right|_p^2}\right)^{\frac{p}{p-2}}= \left(J(\psi_{R})\right)^{\frac{p}{p-2}},$$
\noi
Moreover, using the following strict inequality in the last steps of the proof of Proposition \ref{enerest},
$$
\mbox{max} \, \frac{t^{2}+(1-t)^{2}}{\left(t^{p}+(1-t)^{p}\right)^{\frac{2}{p}}} < 2^{\frac{p-2}{p}}, \quad \forall \ t \in ([0, 1]\setminus N(\frac{1}{2}))$$\text{where $N(\frac{1}{2})$ denotes a neighbourhood of $\frac{1}{2}$}, it can be easily shown that for $R$ sufficiently large
\begin{equation}
\max \left\{J(\psi_R): s \in[0,1], y \in \Sigma\right\}< S_{2, \lambda}. \label{ener1}
\end{equation}
Also, from \cite{MS}, it is known that $S_{1, \lambda}$ is achieved by $w$, which is a solution of \eqref{infprob}. This in turn implies $S_{1, \lambda}=\left(\|w\|_{\lambda}^{2}\right)^{\frac{p-2}{p}}$ and $S_{2, \lambda}=(2 \|w\|_{\lambda}^{2})^{\frac{p-2}{p}}$. \\
Further, $$m= E_{\infty}(w)= \left(\frac{1}{2}-\frac{1}{p}\right)\|w\|_{\lambda}^{2}= \frac{1}{2}\left(\frac{1}{2}-\frac{1}{p}\right)S_{2, \lambda}^{\frac{p}{p-2}}.$$\\
Then by \eqref{3ab}, \eqref{ener1} and using $\left(\frac{1}{2}-\frac{1}{p}\right)S_{2, \lambda}^{\frac{p}{p-2}}= 2m$, the lemma holds.
\end{proof}
\begin{cor} \label{corb}
    There exist $\bar{R}, \bar{\varepsilon}>0$ such that for any $R>\bar{R}$ and for any $\varepsilon \in$ $(0, \bar{\varepsilon})$
\begin{equation*}
\mathcal{A}_{\varepsilon, R}=\max \left\{E_{\varepsilon}\left(t_{R, s, y} \psi_R[s, y]\right): s \in[0,1], y \in \Sigma\right\}<2 m_{\varepsilon}.
\end{equation*}
\begin{proof}
    It results directly from the Lemmas \ref{asym} and \ref{enerlem}. Indeed, using Lemma \ref{enerlem}, we have an existence of $\bar{R}>0$ such that $\mathcal{A}_{\varepsilon, R} < 2m$ for any $R> \bar{R}$. Also, from \eqref{a1.1} we have $m_{\varepsilon} \leq m$. Now if $\mathcal{A}_{\varepsilon, R} > 2m_{\varepsilon} \; \forall \varepsilon > 0$, then $|m- m_{\varepsilon}|> m- \frac{\mathcal{A}_{\varepsilon, R}}{2}\; \forall \varepsilon > 0 $ which contradicts Lemma \ref{asym}.
\end{proof} 
\end{cor}

\medskip


\subsection{Barycentric map} We now introduce a barycentric type function as follows:
for $u \in H^1(\bn) \backslash\{0\}$, we set
\begin{equation*}
\mu(u)(x)=\frac{1}{\left|B_1(0)\right|} \int_{B_1(x)}|u(y)|\mathrm{~d} V_{\mathbb{B}^{N}}(y) \quad x \in \mathbb{B}^N,
\end{equation*}
and we observe that $\mu(u)$ is bounded and continuous, allowing us to introduce the function
\begin{equation*}
\hat{u}(x)=\left[\mu(u)(x)-\frac{1}{2} \max \mu(u)\right]^{+} \quad x \in \bn.
\end{equation*}
The function defined above is continuous and has compact support. Consequently, we can set $\beta: H^1\left(\bn\right) \backslash\{0\}$ $\rightarrow \mathbb{R}^N$ as
\begin{equation*}
\beta(u)=\frac{1}{|\hat{u}|_1} \int_{\mathbb{B}^N}\frac{x}{1+ |x|} \hat{u}(x)  \mathrm{~d} V_{\mathbb{B}^{N}}.
\end{equation*}
Further, define 
\begin{equation*}
C_0=\inf \{E(u): u \in \mathcal{N}, \beta(u)=0\}, \, C_{0, \varepsilon}=\inf \left\{E_{\varepsilon}(u): u \in \mathcal{N}_{\varepsilon}, \,\beta(u)=0\right\} .
\end{equation*}

\medskip


\begin{lem} \label{partlem}
The following facts hold:
\begin{itemize}

\item[$(a)$]  $C_0>m$; 

\medskip

\item[$(b)$] $\displaystyle\lim _{\varepsilon \rightarrow 0} C_{0, \varepsilon}=C_0$.

\end{itemize}
\end{lem}
\begin{proof}
We shall first prove the inequality $(a)$. Clearly, Proposition \ref{prop2ab} gives $C_0 \geq m$. If possible, let us assume that $C_0=m$ and we will obtain a contradiction. Let $\left\{u_n\right\}_n$ be a sequence in $\mathcal{N}$ with $\beta\left(u_n\right)=0$ such that $E\left(u_n\right) \rightarrow m$ and $t_n>0$ be such that $t_n u_n \in \mathcal{N}_{\infty}, \forall n \in \mathbb{N}$. Using the assumption that $a(x) \leq 1$ a.e. in $\bn$, we deduce that
\begin{equation}
m \leq E_{\infty}\left(t_n u_n\right) \leq E\left(t_n u_n\right) \leq E\left(u_n\right)=m+o(1),\label{2cd}
\end{equation}
 which in turn implies that $\left\{t_n u_n\right\}_n$ is a minimizing sequence for $E_{\infty}$ on $\mathcal{N}_{\infty}$. Therefore there exists a sequence $\left\{y_n\right\}_n$ in $\bn$ such that

\begin{equation*}
t_n u_n(x)=w\left(\tau_{-y_{n}}(x)\right)+\phi_n(x), \quad \phi_n \rightarrow 0 \quad \text { strongly in } H^1\left(\bn \right),
\end{equation*}

where $w$ denotes the solution of $(P_{\infty})$ and is  radially symmetric with respect
to origin. Moreover, this sequence $\left\{y_n\right\}_n$ must be bounded because, if (up to a subsequence)
$\lim _{n \rightarrow \infty} d(y_n,0)=+\infty$, i.e., $\left|y_{n}\right| \rightarrow 1$ as $n \rightarrow \infty$ then

\begin{equation}\label{compute_trans}
\lim _{n \rightarrow \infty}\left|\beta\left(t_{n}u_n\right)-\frac{y_n }{1+|y_n|}\right|=0,
\end{equation}
\noi
which will give us to a contradiction, as $\beta\left(t_{n}u_n\right)= \beta\left(u_n\right) =0$ for all $n \in \mathbb{N}$. To prove \eqref{compute_trans}, consider 
\begin{align} 
    \left|\beta\left(t_{n}u_n\right)-\frac{y_n }{1+|y_n|}\right|& =\frac{1}{|\hat{w}|_1} \left|\int_{\bn}\left(\frac{\tau_{y_n}(z)}{1+|\tau_{y_n}(z)|}- \frac{y_{n}}{1+|y_n|}\right)\mathrm{~d} V_{\mathbb{B}^{N}}(z)\right|+o(1) \notag\\ 
    &=\frac{1}{|\hat{w}|_1} \left|\int_{\bn}\left(\frac{\tau_{y_n}(z)+ |y_n|\tau_{y_n}(z)-y_n-y_n |\tau_{y_n}(z)|}{(1+|\tau_{y_n}(z)|)(1+|y_n|)}\right)\mathrm{~d} V_{\mathbb{B}^{N}}(z)\right|+o(1). \notag\\  \label{trancal}
\end{align}
Moreover, it follows from the expression of the hyperbolic translation \eqref{hypt} that $\tau_{y_n}(x) = y_{n} + \circ(1),$ as $|y_n| \rightarrow 1,$ for each fixed $x \in \bn.$ Now using this along with dominated convergence theorem in \eqref{trancal}, \eqref{compute_trans} follows.\\

Thus, up to a subsequence, $y_n \rightarrow \bar{y}$ for some $\bar{y} \in \bn$; indeed, $\bar{y}=0$ because $w$ has radial symmetry with respect to origin. Hence $t_n u_n \rightarrow w$ strongly in $H^1\left(\bn \right)$. Moreover, $a(x) \not \equiv 1$ and 
considering \eqref{2cd}, we obtain
\begin{equation*}
m=E_{\infty}(w)<E(w)=\lim _{n \rightarrow \infty} E\left(t_n u_n\right) \leq \lim _{n \rightarrow \infty} E\left(u_n\right)=m,
\end{equation*}
hence a contradiction. This proves $(a)$.


\medskip

We shall now prove ($b$). Let $\varepsilon>0$ be fixed and for every $\eta>0,$ let $u_\eta \in \mathcal{N}$ be such that $\beta\left(u_\eta\right)=0$ and $E\left(u_\eta\right) \leq C_0+\eta$. Also, let $s_\eta>0$ be such that $s_\eta u_\eta \in \mathcal{N}_{\varepsilon}$. Then
\begin{equation*}
C_{0, \varepsilon} \leq E_{\varepsilon}\left(s_\eta u_\eta\right) \leq E\left(s_\eta u_\eta\right) \leq E\left(u_\eta\right) \leq C_0+\eta.
\end{equation*}
Since $\eta$ is arbitrary, we get
\begin{equation}
C_{0, \varepsilon} \leq C_0 \quad \forall \varepsilon>0 . \label{3ef}
\end{equation} 
Let $v_{\varepsilon} \in \mathcal{N}_{\varepsilon}$ be such that $\beta\left(v_{\varepsilon}\right)=0$ and $E_{\varepsilon}\left(v_{\varepsilon}\right) \leq C_{0, \varepsilon}+\varepsilon$, and let $t_{\varepsilon}>0$ such that $t_{\varepsilon} v_{\varepsilon} \in \mathcal{N}$. Thus
\begin{equation}
\begin{aligned}
C_0 & \leq E\left(t_{\varepsilon} v_{\varepsilon}\right)=E_{\varepsilon}\left(t_{\varepsilon} v_{\varepsilon}\right)+\frac{\varepsilon}{2^*}\left|t_{\varepsilon} v_{\varepsilon}\right|_{2^*}^{2^*} \\
& \leq E_{\varepsilon}\left(v_{\varepsilon}\right)+\frac{\varepsilon}{2^*}\left|t_{\varepsilon} v_{\varepsilon}\right|_{2^*}^{2^*} \\
& \leq C_{0, \varepsilon}+\varepsilon+\frac{\varepsilon}{2^*} t_{\varepsilon}^{2^*}\left|v_{\varepsilon}\right|_{2^*}^{2^*}.
\end{aligned} \label{4aab}
\end{equation}
Furthermore, taking into account \eqref{3ef} gives
\begin{equation*}
E_{\varepsilon}\left(v_{\varepsilon}\right)=\left(\frac{1}{2}-\frac{1}{p}\right)\left\|v_{\varepsilon}\right\|_\la^2+\varepsilon\left(\frac{1}{p}-\frac{1}{2^*}\right)\left|v_{\varepsilon}\right|_{2^*}^{2^*} \leq C_{0, \varepsilon}+\varepsilon \leq C_0+\varepsilon
\end{equation*}
yielding $\left|v_{\varepsilon}\right|_{2^*}^{2^*}$ is bounded. Moreover, taking into account $a(x) \leq 1$ a.e. in $\bn$ and $v_{\varepsilon} \in \mathcal{N}_{\varepsilon}$, and applying Sobolev inequalities, we can deduce

$$\left\|v_{\varepsilon}\right\|_\la^2 [1- S_{N, p}^{-p} \left\|v_{\varepsilon}\right\|_\la^{p-2}-\varepsilon S_{N, 2^*}^{-2^*}\left\|v_{\varepsilon}\right\|_\la^{2^*-2}]\leq 0.$$
Then we can get a $c>0$ independent of small $\varepsilon$ such that
$\|v_{\varepsilon}\|_\la \geq c >0$. Thus we obtain $\left\|v_{\varepsilon}\right\|_\la \not \rightarrow 0$. Hence, performing similar calculations as to obtain \eqref{neq}, we can get that $\left|a ^{1/p} v_{\varepsilon}\right|_p \not \rightarrow 0$. Further, $t_{\varepsilon} v_{\varepsilon} \in \mathcal{N}$ $\implies t_{\varepsilon} = \left(\frac{\left\|v_{\varepsilon}\right\|_\la}{\left|a ^{1/p} v_{\varepsilon}\right|_p}\right)^{\frac{1}{p-2}}$ giving $\left\{t_{\varepsilon}\right\}$ is bounded. Thus, using \eqref{4aab}, we get $\liminf _{\varepsilon \rightarrow 0} C_{0, \varepsilon} \geq C_0$ that, along with \eqref{3ef}, gives ($b$).
\end{proof}


\medskip

\begin{cor} \label{cora}
There exists $\widetilde{\varepsilon}>0$ such that the inequality $C_{0, \varepsilon}>\frac{C_0+m}{2}$ is true for every $\varepsilon \in(0, \widetilde{\varepsilon})$ 
\end{cor}
\begin{proof}
   The statement can be easily proved by appropriately using $(a)$ and $(b)$ of Lemma \ref{partlem}. 
\end{proof}

\medskip


\begin{lem}\label{lema}
    Let $\mathcal{A}_{\varepsilon, R}$ be as in Lemma \ref{enerlem}. Then $\widehat{R}>0$ exists such that $C_{0, \varepsilon} \leq$ $\mathcal{A}_{\varepsilon, R}\; \; \forall R>\widehat{R}, \forall \varepsilon>0$.
\end{lem} 
\begin{proof}
     Using the similar calculations performed in \eqref{trancal}, we can deduce that for $y \in \Sigma$,
\begin{equation*}
\beta(w\left(\tau_{-y}(\cdot)\right))=\frac{y}{1+|\tau_{-x_{1}}(y)|}+o(1) \text {, }
\end{equation*}
where $o(1) \rightarrow 0$ as $R^{\prime} \rightarrow \infty$.
Then following the arguments similar to in \cite[Section~5]{DDS} and defining 
 $J: \bar B_{R^{\prime}}(x_{1}) \rightarrow \rn $ by $$K(x):=\frac{x}{1+ \tanh{\frac{R^{\prime}}{2}}},$$ 
		we have for $R^{\prime}$ large enough, and $\forall y \in \Sigma =\partial B_{R^{'}}\left(x_1\right) $, \;$\beta \circ \psi_{R^{\prime}}[1, y] =K(y)$. 
		Therefore, applying the invariance of the topological degree by homotopy, and the solution property of degree, we can ensure the existence of some  $(s_R, y_R) \in [0,1] \times \Sigma$ such that $\beta(\psi_{R}[s_R, y_R])=0$.
  Thus $\beta\left(t_{R,s_R, y_R} \psi_{R}\left[s_R, y_R\right]\right)=0$. Since $t_{R, s_R, y_R} \psi_{R}\left[s_R, y_R\right] \in$ $\mathcal{N}_{\varepsilon}$, the assertion follows.
\end{proof}

\medskip


\begin{lem} \label{lemb}
Assume $\tilde{\varepsilon}$ as in Corollary \ref{cora} and $\varepsilon \in(0, \tilde{\varepsilon})$. There exists $\tilde{R}>0$ such that for any $R>\tilde{R}$
\begin{equation*}
\mathcal{B}_{\varepsilon, R}:=\max \left\{E_{\varepsilon}\left(t_{R, 1, y} \psi_R [1, y]\right): y \in \Sigma\right\}<C_{0, \varepsilon} .
\end{equation*}
\end{lem}		
\begin{proof}
    For simplicity denote $t_R=t_{R, 1, y}$ and $\psi_R=\psi_R[1, y]$. We argue by contradiction so let us assume that there exist $R_n \rightarrow \infty$ and $y_n \in \Sigma$ such that $E_{\varepsilon}\left(t_{R_n} \psi_{R_n}\right) \geq C_{0, \varepsilon}$ for every $n \in \mathbb{N}$. Since $t_{R_n} \psi_{R_n} \in \mathcal{N}_{\varepsilon}$ we can write
\begin{equation}
\begin{aligned}
& E_{\varepsilon}\left(t_{R_n} \psi_{R_n}\right)=\left(\frac{1}{2}-\frac{1}{p}\right)\left\|t_{R_n} \psi_{R_n}\right\|_\la^2+\varepsilon\left(\frac{1}{p}-\frac{1}{2^*}\right)\left|t_{R_n} \psi_{R_n}\right|_{2^*}^{2^*} \\
& =\left(\frac{1}{2}-\frac{1}{p}\right) t_{R_n}^2\left\|w\left(\tau_{-y_{n}}(\cdot)\right)\right\|_\la^2+\varepsilon\left(\frac{1}{p}-\frac{1}{2^*}\right) t_{R_n}^{2^*}\left|w\left(\tau_{-y_{n}}(\cdot)\right)\right|_{2^*}^{2^*}. \label{3sa}
\end{aligned}
\end{equation}
We can notice that in our setting $0<m \leq C_{0, \varepsilon} \leq E_{\varepsilon}\left(t_{R_n} \psi_{R_n}\right) \leq \mathcal{A}_{\varepsilon, R}<2 m$ and that $0<c \leq\left\|w\left(\tau_{-y_{n}}(\cdot)\right)\right\|_\la^2 \leq C<\infty, \forall n \in \mathbb{N}$. Thus $0<c_1 \leq t_{R_n} \leq C_1<\infty$ follows from \eqref{3sa}. So, up to a subsequence, we can assume $t_{R_n} \rightarrow t>0$. Since $R_n \rightarrow \infty$, the same estimates provided in the energy estimates of \cite{DDS} help us prove $E_{\varepsilon}\left(t_{R_n} \psi_{R_n}\right) \rightarrow E_{\varepsilon, \infty}(t w)$, and we get
\begin{equation*}
\begin{aligned}
C_{0, \varepsilon} & \leq E_{\varepsilon, \infty}(t w)=E_{\infty}(t w)-\frac{\varepsilon}{2^*}|t w|_{2^*}^{2^*} \\
& \leq E_{\infty}(w)-\frac{\varepsilon}{2^*}|t w|_{2^*}^{2^*} \\
& =m-\frac{\varepsilon}{2^*}|t w|_{2^*}^{2^*}<m
\end{aligned}
\end{equation*}
which gives rises to a contradiction taking into account Corollary \ref{cora} and Lemma \ref{partlem}$(a)$.
\end{proof}

\medskip


\textbf{Proof of Theorem \ref{aless}}
Firstly, we recollect all the values that have been stated and used in obtaining the previous few results:
\begin{equation*}
\begin{gathered}
\mathcal{A}_{\varepsilon, R}=\max \left\{E_{\varepsilon}\left(t_{R, s, y} \psi_R[s, y]\right): s \in[0,1], y \in \Sigma\right\},\\
\mathcal{B}_{\varepsilon, R}=\max \left\{E_{\varepsilon}\left(t_{R, 1, y} \psi_R[1, y]\right): y \in \Sigma\right\}, \\
C_{0, \varepsilon}=\inf \left\{E_{\varepsilon}(u): u \in \mathcal{N}_{\varepsilon}, \beta(u)=0\right\}.
\end{gathered}
\end{equation*}

\medskip

By Corollaries \ref{corb} and \ref{cora}, and Lemmas \ref{enerlem}, \ref{partlem}, \ref{lema} and \ref{lemb}, the following inequalities
\begin{equation}
\begin{cases}
 (a) & \mathcal{B}_{\varepsilon, R}<C_{0, \varepsilon} \leq \mathcal{A}_{\varepsilon, R} \\ 
(b) & m<\frac{c_0+m}{2}<C_{0, \varepsilon} \leq \mathcal{A}_{\varepsilon, R} \leq \mathcal{A}<2 m \\
(c) & \mathcal{A}_{\varepsilon, R}<2 m_{\varepsilon} 
\end{cases} \label{ineq}
\end{equation}
hold true for all $R>\max \{\bar{R}, \widetilde{R}, \widehat{R}\}$ and for all $0<\varepsilon<\min \{\bar{\varepsilon}, \widetilde{\varepsilon}\}$. Assume $\delta$ such that $0<$ $\delta<\min \left\{\frac{m}{2}, 2 m-\mathcal{A}, \frac{C_0-m}{2}\right\}$. Further, we take $\varepsilon_\delta$ according to Proposition \ref{propps}.

\medskip

Finally, to prove Theorem \ref{aless}, we claim that $E_{\varepsilon}$ constrained on $\mathcal{N}_{\varepsilon}$ has a (PS)-sequence in $\left[C_{0, \varepsilon}, \mathcal{A}_{\varepsilon, R}\right]$ for every $0<\varepsilon<\widehat{\varepsilon}:=\min \left\{\varepsilon_\delta, \bar{\varepsilon}, \widetilde{\varepsilon}\right\}$. Having done this, the Proposition \ref{propps} will guarantee the existence of a non-zero critical point $\bar{u}$ with $E_{\varepsilon}(\bar{u}) \leq \mathcal{A}_{\varepsilon, R}$. 

\medskip

We argue by contradiction, so let us suppose that there is no (PS)-sequence in the interval $\left[C_{0, \varepsilon}, \mathcal{A}_{\varepsilon, R}\right]$. Then, standard deformation arguments give the existence of $\eta>0$ such that the sublevel $E_{\varepsilon}^{C_{0, \varepsilon}-\eta}:=\left\{u \in \mathcal{N}_{\varepsilon}: E_{\varepsilon}(u) \leq C_{0, \varepsilon}-\eta\right\}$ is a deformation retract of the sublevel $E_{\varepsilon}^{\mathcal{A}_{\varepsilon, R}}:=\left\{u \in \mathcal{N}_{\varepsilon}: E_{\varepsilon}(u) \leq \mathcal{A}_{\varepsilon, R}\right\}$, i.e., there exists a continuous function $\sigma: E_{\varepsilon}^{\mathcal{A}_{\varepsilon, R}} \rightarrow E_{\varepsilon}^{C_{0, \varepsilon}-\eta}$ such that
\begin{equation}
\sigma(u)=u \quad \text { for any } u \in E_{\varepsilon}^{C_{0, \varepsilon}-\eta}.
\end{equation}
Moreover, taking into account the inequality \ref{ineq} $(a)$, $\eta$ can be chosen so small that
\begin{equation}
C_{0, \varepsilon}-\eta>\mathcal{B}_{\varepsilon, R}. \label{3df}
\end{equation}
Let us define the map $\mathcal{H}:[0,1] \times \Sigma \rightarrow \mathbb{R}^N$ by
\begin{equation}
\mathcal{H}(s, y)=\beta\left(\sigma\left(t_{R, s, y} \psi_R[s, y]\right)\right). \label{beta}
\end{equation}
Applying \eqref{3df}, \eqref{beta}, and the arguments similar to that established in Lemma \ref{lema}, we can find a point $(\tilde{s}, \tilde{y}) \in[0,1] \times \Sigma$ for which
\begin{equation*}
0=\mathcal{H}(\tilde{s}, \tilde{y})=\beta\left(\sigma\left(t_{R, \tilde{s}, \tilde{y}} \psi_R[\tilde{s}, \tilde{y}]\right)\right) .
\end{equation*}
Then, $E_{\varepsilon}\left(\sigma\left(t_{R, \tilde{s}, \tilde{y}} \psi_R[\tilde{s}, \tilde{y}]\right)\right) \geq C_{0, \varepsilon}$ in contrast to $\sigma\left(t_{R, s, y} \psi_R[s, y]\right) \in E_{\varepsilon}^{C_{0, \varepsilon}-\eta}$ for every $(s, y) \in[0,1] \times \Sigma$, so the claim must hold true.

Let the critical point, which we have found out, be $\bar{u} \in E_{\varepsilon}^{\mathcal{A}_{\varepsilon, R}}$. In order to prove that $\bar{u}$ is a constant sign function, assume, by contradiction, that $\bar{u}=\bar{u}^{+}-\bar{u}^{-}$, with $\bar{u}^{\pm} \not\equiv 0$. We determine that $\bar{u}^{\pm} \in \mathcal{N}_{\varepsilon}$ by multiplying the equation in \eqref{mainprob} by $\bar{u}^{\pm}$, so
\begin{equation*}
E_{\varepsilon}(\bar{u})=E_{\varepsilon}\left(\bar{u}^{+}\right)+E_{\varepsilon}\left(\bar{u}^{-}\right) \geq 2 m_{\varepsilon},
\end{equation*}
contrary to \ref{ineq}(c).

\medskip


\par\bigskip\noindent
	\textbf{Acknowledgments.}
	D.~Ganguly is partially supported by the INSPIRE faculty fellowship (IFA17-MA98).  
	D.~Gupta is supported by the PMRF. 
	
	\medskip 
		
	\noindent \textbf{Conflict of interest:} All authors certify that there is no actual or potential conflict of interest about this article.


\end{document}